\documentclass[notitlepage]{amsart}

\usepackage{amsfonts}
\usepackage
{inputenc}
\usepackage[all]{xy}
\usepackage{graphicx}
\usepackage{tikz}
\usetikzlibrary{arrows,decorations.pathmorphing,decorations.pathreplacing,positioning,shapes.geometric,shapes.misc,decorations.markings,decorations.fractals,calc,patterns}

\newtheorem{theorem}{Theorem}[subsection]
\newtheorem{corollary}[theorem]{Corollary}
\newtheorem{definition}[theorem]{Definition}
\newtheorem{example}[theorem]{Example}

\newtheorem{lemma}[theorem]{Lemma}

\newtheorem{proposition}[theorem]{Proposition}

\newtheorem{remark}[theorem]{Remark}

\newtheorem{algorithm}[theorem]{Algorithm}

\newcommand{\pn}{\par\noindent}

\newcommand{\mor}[3]{$\xymatrix@1@C=15pt{#3: #1\ar[r]& #2}$}
\newcommand{\iso}[3]{$\xymatrix@1@C=15pt{#3: #1\ar[r]^-{\cong}& #2}$}

\usepackage{latexsym,amssymb,amscd}
\usepackage{amsmath}
\usepackage[all]{xy}

\begin{document}
\sloppy
\bibliographystyle{plain}

\title[Infinite friezes and triangulations of the strip]{Infinite friezes and triangulations of the strip}

\author{D. Smith}
\address{\pn David Smith; Department of mathematics,
Bishop's University. Sherbrooke, Qu\'ebec, Canada,  J1M1Z7.}
\email{dsmith@ubishops.ca}

\begin{abstract} 
The infinite friezes of positive integers were introduced by Tschabold as a variation of the classical Conway-Coxeter frieze patterns. These infinite friezes were further shown be to realizable via triangulations of the infinite strip by Baur, Parsons and Tschabold.  In this paper, we show that the construction of Baur, Parsons and Tschabold can be slightly adapted in order to obtain a bijection between the infinite friezes and the so-called admissible triangulations of the infinite strip with no special marked points on the upper boundary. As a consequence, we obtain that the infinite friezes with enough ones are in bijection with the admissible triangulations of the infinity-gon.
\end{abstract}

\maketitle

\section*{Introduction}

The study of friezes goes back to the 1970's with the publication of companion papers by Conway and Coxeter \cite{CC73-1, CC73-2}. In these papers, the authors established a bijection between the so-called frieze patterns (of positive integers) of rank $n$ and the triangulations of the regular polygon with $n$ vertices, for $n\geq 3$. Roughly, a frieze pattern of order $n$, with $n\geq 3$, can be seen as an arrangement of real numbers into a band consisting of $n-1$ adjacent infinite diagonals in the discrete plane $\mathbb{Z}\times\mathbb{Z}$ (the first and the last consisting only of ones) such that each adjacent $2\times 2$-matrix in this band has determinant one. These papers were shortly followed by a paper by Broline et al. \cite{BCI74} in which new interpretations of the numbers appearing in the frieze patterns of Conway and Coxeter were given.

The interest for these frieze patterns was revived in the 2000's due to their connections with the cluster algebras and cluster categories via the Caldero-Chapoton map, and many generalizations were then studied.  In particular, the $SL_2$-tilings of the discrete plane $\mathbb{Z}\times\mathbb{Z}$ were introduced by Assem, Reutenauer and Smith in \cite{ARS10}: an $SL_2$-tiling is a map $t:\mathbb{Z}\times\mathbb{Z}\rightarrow \mathbb{N}$ such that $\left|\begin{array}{cc} t(i,j) & t(i, j+1) \\ t(i+1,j) & t(i+1, j+1)\end{array}\right|=1$ for all $(i,j)\in \mathbb{Z}\times\mathbb{Z}$.  In \cite{HJ13}, Holm and J{\o}rgensen demonstrated that the $SL_2$-tilings (of the discrete plane) with enough ones are in bijection with some triangulations of the strip, called good triangulations.  This construction was further extended in \cite{BHJ15}. We refer the reader to the rather complete survey paper by Morier-Genoud \cite{MG15} for a more in-depth immersion in the world of friezes and $SL_2$-tilings. 

More recently, Tschabold introduced in \cite{T15} the notion of infinite friezes.  In \cite{BPT15}, Baur, Parsons and Tschabold showed that every infinite frieze can be realized as an admissible triangulation of the infinite strip with marked points $\mathbb{V}=\mathbb{V}(M_1, M_2)$, where the marked points on the lower boundary are indexed by $M_1=\mathbb{Z}$ and the marked points on the upper boundary are indexed by some subset $M_2$ of $\mathbb{Z}$. Moreover, they showed that every periodic infinite frieze can be realized as a triangulation of the annulus.

The purpose of this paper is two-fold. First, we show that the infinite friezes share many properties with the frieze patterns of Conway and Coxeter \cite{CC73-1,CC73-2}, completing a work that was initiated in \cite{T15}. Second, we show that the constructions of Baur, Parsons and Tschabold can be refined in order to obtain the following result: \\

\noindent \textbf{Theorem.}
There exists a bijection between the infinite friezes and the admissible triangulations of the infinite strip $\mathbb{V}(M_1, M_2)$ having no special marked points on the upper boundary, up to Dehn twist equivalence when $M_2=\mathbb{Z}$.
\\

As a consequence, we obtain an analogue of the main result of Holm and J{\o}rgensen~\cite{HJ13}: we show that the infinite friezes with enough ones are in bijection with the admissible triangulations of the $\infty$-gon, see Proposition~\ref{prop:enough ones}.

The paper is organized as follows: In Section 1, we recall the necessary background on the frieze patterns of Conway and Coxeter \cite{CC73-1, CC73-2} and infinite friezes of Tschabold \cite{T15}, and show that these two objects share many properties. In Section 2, we recall and refine the main ingredients from the paper by Baur, Parsons and Tschabold \cite{BPT15} about the connection between the infinite friezes and the triangulations of the infinite strip. This allows us to present an algorithm showing that every infinite frieze can be realized via an admissible triangulation of the infinite strip with no special marked points on the upper boundary. We also discuss the so-called Dehn twist and show that the Conway-Coxeter counting method \cite{CC73-1, CC73-2} and the Broline-Crowe-Isaacs counting method \cite{BCI74} still hold for the infinite friezes. Finally, Section 3 is devoted to the proof of our main theorem and to the study of the triangulations of the $\infty$-gon.

\section{Frieze patterns and infinite friezes}

In this section, we first recall the necessary background on the frieze patterns of Conway and Coxeter, and especially the their interpretations in terms of triangulated polygons, see \cite{BCI74, CC73-1, CC73-2}. In the second part, we recall the infinite friezes from \cite{BPT15}, and demonstrate some elementary properties of the infinite friezes. Our goal is actually to create parallels with some results obtained by Conway and Coxeter on frieze patterns. 

However, before doing so, the reader should be advised that, in this paper, we took the difficult decision to adopt the following convention from the works of Holm, J{\o}rgensen and Bessenrodt \cite{BHJ15, HJ13}, rather than the ones used by Baur, Parsons and Tschabold in \cite{BPT15, T15}. This decision was motivated by the wish of being more in line with the original notations of Conway and Coxeter (that are slightly different than those used in \cite{BPT15, T15}) and, possibly more importantly, facilitating the location of a specific entry in an infinite frieze or a frieze pattern.  
\\

\noindent\textbf{Convention. }
As in \cite{BHJ15, HJ13}, we follow the matrix convention when dealing with entries in the discrete plane $\mathbb{Z}\times\mathbb{Z}$, so that the $x$-coordinate increases from top to bottom and the $y$-coordinate increases from left to right.  
\\

As a consequence, our frieze patterns and infinite friezes look slightly different than the ones given in the original papers by Conway-Coxeter \cite{CC73-1, CC73-2} and Tschabold \cite{T15}, although there are mathematically equivalent. Roughly speaking, our frieze patterns and infinite friezes can be obtained from the ones in \cite{CC73-1, CC73-2} and \cite{T15} upon a anticlockwise rotation of $135$ degrees. For instance, an infinite frieze in \cite{BPT15, T15} consists of infinitely many offset bi-infinite rows, while our infinite friezes consist of infinitely many bi-infinite diagonals (running \lq northwest\rq\, to \lq southeast\rq) in the discrete plane $\mathbb{Z}\times\mathbb{Z}$.

In addition to these cosmetic differences between our infinite friezes and those of \cite{BPT15, T15}, note that we also added an antisymmetry (see condition (c) in Definition \ref{def:infinite frieze} below), having as a result that an infinite frieze becomes a complete filling of the discrete plane $\mathbb{Z}\times\mathbb{Z}$ with integers, therefore turning them into $SL_2$-tilings of the discrete plane where, however, not all entries are positive integers.  This antisymmetry is obtained for free, and does not alternate the nature of an infinite frieze. However, this artificial addition allows to simplify the writing of some definitions and statements, and allows to refer to results from \cite{HJ13} on $SL_2$-tilings.

\subsection{Conway-Coxeter frieze patterns}\label{section:CC}
Let $n\geq 3$ be an integer.  Consider the subset of $\mathbb{Z}\times\mathbb{Z}$ defined by
\[
B_n=\{(i,j)\in\mathbb{Z}\times\mathbb{Z} \ | \  0<j-i<n \}. 
\]
Considering our matrix notation convention, $B_n$ consists in a band of $n-1$ bi-infinite adjacent diagonals running \lq northwest\rq\, to \lq southeast\rq.  

\begin{definition}
Let $n\geq 3$.  A frieze pattern (of positive integers) of rank $n$ is a function $f:B_n\rightarrow \{1, 2, 3, \dots\}$ such that 
\begin{enumerate}
\item [(a)] $f(i,j)=1$ if $j-i=1$ or $j-i=n-1$, and 
\item[(b)] $\left|\begin{array}{cc} f(i,j) & f(i, j+1) \\ f(i+1,j) & f(i+1, j+1)\end{array}\right|=1$ for all $(i,j) \in B_n$ such that $(i+1,j)\in B_n$ and $(i,j+1)\in B_n$.
\end{enumerate}
\end{definition}
 In other words, a frieze pattern (of positive integers) of rank $n$ consists of an arrangement of positive integers into $n-1$ adjacent diagonals such that the first and the last diagonal contain only $1's$ and all adjacent $2\times 2$-matrix has determinant one.
 
 
\begin{example}\label{ex:seven-gon}
The Figure~\ref{fig:tiling1} presents a part of a frieze pattern of rank $7$.
{\em
\begin{center}
\begin{figure}
  \centering
  \begin{tikzpicture}[auto]
    \draw (2.8,-2.3) -- (2.8,1.4) -- (-1.5,1.4) -- (2.8,-2.3);
    \matrix
    {
			&&\node{{\em (-1)}};&\node{{\em (0)}};&\node{{\em (1)}};&\node{{\em (2)}};&\node{{\em (3)}};&\node{{\em (4)}};&\node{{\em (5)}};& \node{{\em (6)}}; &\node{{\em (7)}};&\node{{\em (8)}};&\node{{\em (9)}};& \\
      &&&&&&& \node{$\vdots$}; &&&&&& \\
			 \node{{\em (-1)}}; &&  \node {-}; & \node {1}; & \node {4}; & \node {3}; & \node {2}; & \node {3}; & \node {1}; & \node {-}; & \node {-};& \node {-};& \node {-};\\
			\node{{\em (0)}}; && \node {-}; &\node {-}; & \node {1}; & \node {1}; & \node {1}; & \node {2}; & \node {1}; & \node {1};& \node {-}; &\node {-}; &\node {-}; & \\
      \node{{\em (1)}}; && \node {-}; &\node {-}; &\node {-}; & \node {1}; & \node {2}; & \node {5}; & \node {3}; & \node {4}; & \node {1}; &\node {-}; &\node {-}; &\\
      \node{{\em (2)}};&& \node {-}; &\node {-}; &\node {-}; & \node {-}; & \node {1}; & \node {3}; & \node {2}; & \node {3}; & \node {1}; &\node {1}; &\node {-}; &\\
      \node{{\em (3)}};&& \node {-}; &\node {-}; &\node {-}; & \node {-}; & \node {-}; & \node {1}; & \node {1}; & \node {2}; & \node {1};&\node {2}; &\node {1}; &\\
      \node{{\em (4)}};&\node{$\cdots$};&\node {-}; &\node {-}; & \node {-}; & \node {-}; & \node {-}; & \node {-}; & \node {1}; & \node {3}; & \node {2};&\node {5}; &\node {3}; &\node{$\cdots$};\\
      \node{{\em (5)}};&& \node {-}; &\node {-}; &\node {-}; & \node {-}; & \node {-}; & \node {-}; & \node {-}; & \node {1}; & \node {1};&\node {3}; &\node {2}; &\\
      \node{{\em (6)}};&& \node {-}; &\node {-}; &\node {-}; & \node {-}; & \node {-}; & \node {-}; & \node {-}; & \node {-}; & \node {1};&\node {4}; &\node {3}; &\\
      \node{{\em (7)}};&& \node {-}; &\node {-}; &\node {-}; & \node {-}; & \node {-}; & \node {-}; & \node {-}; & \node {-}; & \node {-}; &\node {1}; &\node {1}; &\\
      \node{{\em (8)}};&& \node {-}; &\node {-}; &\node {-}; & \node {-}; & \node {-}; & \node {-}; & \node {-}; & \node {-}; & \node {-};&\node {-}; &\node {1}; &\\
      &&&&&&& \node{$\vdots$}; &&&&&& \\
    };

  \end{tikzpicture} 
  \caption{Frieze pattern of rank $7$}
\label{fig:tiling1}
\end{figure}
\end{center}
}
 \end{example}

In \cite{CC73-1, CC73-2}, Conway and Coxeter outlined the fact that the frieze patterns of rank $n$ are in bijection with the triangulations of the regular $n$-gon, with $n\geq 3$.  Recall briefly that in this context, a triangulation can be seen as a maximal collection of noncrossing line segments (called arcs) joining two vertices of the regular $n$-gon.

The bijection is given via the following procedure, that we will refer to as the \textbf{CC-counting method}.  Suppose that a triangulation $T$ of the $n$-gon is given. Choose a vertex $A$ and label it zero.  Every vertex connected to $A$ by an arc of the triangulation or by a side of the polygon is labeled one.  Then, whenever a triangle has exactly two labeled vertices, the third vertex is labeled by the sum of the labels of the other two vertices.  The number $CC_T(A,B)$ will denote the resulting label on the vertex $B$.  Now, if the vertices are labeled from $1$ to $n$, the function $f:\{(A,B) \ | 1\leq A< B \leq n\}\rightarrow \{1, 2, 3, \dots\}$ given by $f(A,B)=CC_T(A,B)$ defines the fundamental region of a frieze pattern $f_T:B_n\rightarrow \{1, 2, 3, \dots\}$ of rank $n$ obtained upon translation of this fundamental region to the left (northwest) and right (southeast) by means of glide reflections as illustrated below (see \cite{CC73-1, CC73-2} for more details).
\vspace{-0.5cm}
\[
  \xymatrix @-3.25pc @! {
      & & *{} \\
      & & & \\
      \ddots & & & & \\
      & & & & & \\
      & & & & & & \\
      & & & & & & & & & & & & & & \\
      & & & & & & & & & & & & & & \\
      & & *{} \ar@{-}[rrrrrrr] \ar@{-}[uuuuuuu] & & & & & & & *{} \ar@{-}[uuuuuuulllllll] & & & & & \\
      & & & *{} \ar@{-}[rrrrrrr] \ar@{-}[dddddddrrrrrrr] & & & & & & & *{} \ar@{-}[ddddddd] & & & & \\
      & & & & & & & & & & & *{} & & & & \\
      & & & & & & & & & & & & & & & \\
      & & & & & & & & & & & & & & \\
      & & & & & & & & & & & & & & \\
      & & & & & & & & & & & & & & & \\
      & & & & & & & & & & & & & & & & \\
      & & & & & & & & & & & & & & & & & \\
      & & & & & & & & & & & *{} \ar@{-}[uuuuuuu] \ar@{-}[rrrrrrr] & & & & & & & *{} \ar@{-}[uuuuuuulllllll] \\
      & \\
      & & & & & & & & & & & & & & & & \ddots
    }
\]
More explicitly, if, for each $r\in \mathbb{Z}$, we denote by $\overline{r}$ the unique integer satisfying $1\leq \overline{r}\leq n$ and $\overline{r}\equiv r \pmod n$, then the function $f_T:B_n\rightarrow \{1, 2, 3, \dots\}$ defined by 
\[
f_T(i,j)=\left\{\begin{array}{ll}
CC_T(\overline{i}, \overline{j}), & \text{if } i<j,\\
CC_T(\overline{j}, \overline{i}), & \text{if } j<i.
\end{array} \right.
\]
is the frieze pattern (of positive integers) of rank $n$ associated to $T$.

\begin{example}\label{ex:seven-gon 2}
The following triangulation $T$ of the regular $7$-gon corresponds to the frieze pattern of rank $7$
given in Example \ref{ex:seven-gon}:
%
\begin{center}
\includegraphics[scale=1]{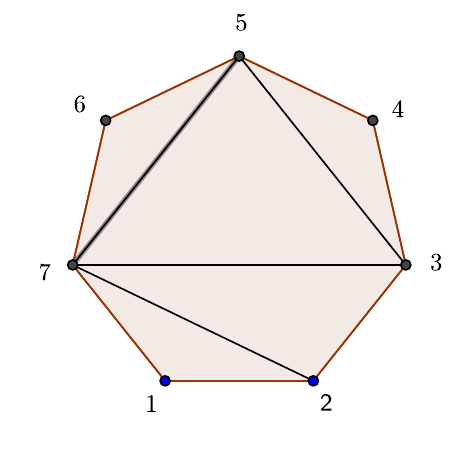}
\end{center}
For instance, one can verify that $CC_T(1,2)=1$, $CC_T(1,3)=2$, $CC_T(1,4)=5$, $CC_T(1,5)=3$, $CC_T(1,6)=4$ and $CC_T(1,7)=1$, and the sequence $1, 2, 5, 3, 4, 1$ appears on the row $(1)$ of the frieze pattern.  Similarly, $CC_T(2,3)=1$, $CC_T(2,4)=3$, $CC_T(2,5)=2$, $CC_T(2,6)=3$ and $CC_T(2,7)=1$, and the sequence $1, 3, 2, 3, 1$ appears on the row $(2)$ of the frieze pattern.  Finally, observe also that the column $(8)$ is the same as the row $(1)$ due to the glide reflection explained above: indeed, $f_T(x,8)=CC_T(1,x)$ for all $x\in\mathbb{Z}$ satisfying $2\leq x\leq 7$.
\end{example}

Another interpretation of the numbers appearing in a frieze pattern was provided by Broline, Crowe and Isaacs in \cite{BCI74}.  We will refer to this as the \textbf{BCI-counting method}. Suppose that a triangulation $T$ of the $n$-gon, with $n\geq 3$, is given. Given two vertices $A$ and $B$, consider the two walks (clockwise and counterclockwise) on the boundary arcs from $A$ to $B$, say $A,P_1,P_2,\dots, P_r,B$ and $A,Q_1,Q_2,\dots, Q_s, B$.  A \textbf{$r$-tuple for the walk} $A,P_1,P_2,\dots, P_r,B$ is an ordered $r$-tuple of triangles in the triangulation $T$ such that the $i$-th entry is a triangle of the triangulation having $P_i$ as a vertex, and no triangle appears more than once in the $r$-tuple.  Denote by $BCI_T(A,P_1,P_2,\dots, P_r,B)$ the number of $r$-tuples for the walk $A,P_1,P_2,\dots, P_r,B$, with the convention that $BCI_T(A)=0$ if $A=B$ and $BCI_T(A,B)=1$ if $A,B$ is a walk. Similarly, we define $BCI_T(A,Q_1,Q_2,\dots, Q_s,B)$.
 
The following proposition gathers two important results from \cite{BCI74}.
\begin{proposition}\label{prop:BCI}
Let $T$ be a triangulated $n$-gon, with $n\geq 3$.  Let $A$ and $B$ be two vertices of $n$-gon and $AP_1P_2\dots P_rB$ and $A,Q_1,Q_2,\dots, Q_s, B$ be the two walks from $A$ to $B$. Then 
\[
BCI_T(A,P_1,P_2,\dots, P_r,B)=CC_T(A,B)=BCI_T(A,Q_1,Q_2,\dots, Q_s,B).
\]
\end{proposition} 
As a direct consequence, we get the following result that appeared in \cite{CC73-1,CC73-2} (compare with Examples~\ref{ex:seven-gon} and \ref{ex:seven-gon 2}). As we will see in Section~\ref{section:BPT}, this is precisely the idea that was adopted by Baur, Parsons and Tschabold in order to associate an infinite frieze to any triangulation of the infinite strip.

\begin{corollary}\label{cor:CC_quiddity}
Let $T$ be a triangulation of the regular $n$-gon, where $n\geq 3$. Let the vertices of the $n$-gon be denoted by $1, 2, 3\dots, n$. Then, for every $i\in\{1, 2, 3\dots, n\}$, we have that $CC_T(i,i+2)$ corresponds to the number of triangles incident to vertex $i+1$, where $i, i+1$ and $i+2$ are taken modulo $n$.  Consequently, the frieze pattern associated to $T$ is completely determined by the number of triangles incident to each vertex of $T$.
\end{corollary}

\subsection{Infinite friezes}

\begin{definition}\label{def:infinite frieze}
An \textbf{infinite frieze} is a map $t:\mathbb{Z}\times\mathbb{Z}\rightarrow \mathbb{Z}$ such that:
\begin{enumerate}
\item[(a)] $t(i,j)=0$ if $i=j$,
\item[(b)] $t(i,j)\geq 1$ if $i<j$, and $t(i,j)=1$ if $i=j-1$, 
\item[(c)] $t(i,j)=-t(j,i)$ for all $(i,j)\in \mathbb{Z}\times\mathbb{Z}$,
\item[(d)] The \textbf{unimodular rule} holds, that is, $\left|\begin{array}{cc} t(i,j) & t(i, j+1) \\ t(i+1,j) & t(i+1, j+1)\end{array}\right|=1$ for all $(i,j)\in \mathbb{Z}\times\mathbb{Z}$.
\end{enumerate}
\end{definition}

Consequently, an infinite frieze can be interpreted as a bi-infinite antisymmetric matrix of integers satisfying the unimodular rule such that the first diagonal above the main diagonal contains only $1$'s, and in which the number $t(i,j)$ is located at the intersection of the $i$-th row and the $j$-th column.

\begin{example}\label{ex:tilings}
Here is an obvious example of an infinite frieze, in which we labeled the rows and columns for the convenience of the reader.
 {\em
\begin{center}
  \begin{tikzpicture}[auto]

    \matrix
    {
			&&\node{{\em (-5)}};&\node{{\em (-4)}};&\node{{\em (-3)}};&\node{{\em (-2)}};&\node{{\em (-1)}};&\node{{\em (0)}};&\node{{\em (1)}};& \node{{\em (2)}}; &\node{{\em (3)}};&\node{{\em (4)}};&\node{{\em (5)}};& \\
      &&&&&&& \node{$\vdots$}; &&&&&& \\
			 \node{{\em (-5)}}; && \node {0}; & \node {1}; & \node {2}; & \node {3}; & \node {4}; & \node {5}; & \node {6}; & \node {7};& \node {8};& \node {9};& \node {10};\\
			\node{{\em (-4)}}; && \node {-1}; &\node {0}; & \node {1}; & \node {2}; & \node {3}; & \node {4}; & \node {5}; & \node {6};& \node {7}; &\node {8}; &\node {9}; & \\
      \node{{\em (-3)}}; && \node {-2}; &\node {-1}; &\node {0}; & \node {1}; & \node {2}; & \node {3}; & \node {4}; & \node {5}; & \node {6}; &\node {7}; &\node {8}; &\\
      \node{{\em (-2)}};&& \node {-3}; &\node {-2}; &\node {-1}; & \node {0}; & \node {1}; & \node {2}; & \node {3}; & \node {4}; & \node {5}; &\node {6}; &\node {7}; &\\
      \node{{\em (-1)}};&& \node {-4}; &\node {-3}; &\node {-2}; & \node {-1}; & \node {0}; & \node {1}; & \node {2}; & \node {3}; & \node {4};&\node {5}; &\node {6}; &\\
      \node{{\em (0)}};&\node{$\cdots$};&\node {-5}; &\node {-4}; & \node {-3}; & \node {-2}; & \node {-1}; & \node {0}; & \node {1}; & \node {2}; & \node {3};&\node {4}; &\node {5}; &\node{$\cdots$};\\
      \node{{\em (1)}};&& \node {-6}; &\node {-5}; &\node {-4}; & \node {-3}; & \node {-2}; & \node {-1}; & \node {0}; & \node {1}; & \node {2};&\node {3}; &\node {4}; &\\
      \node{{\em (2)}};&& \node {-7}; &\node {-6}; &\node {-5}; & \node {-4}; & \node {-3}; & \node {-2}; & \node {-1}; & \node {0}; & \node {1};&\node {2}; &\node {3}; &\\
      \node{{\em (3)}};&& \node {-8}; &\node {-7}; &\node {-6}; & \node {-5}; & \node {-4}; & \node {-3}; & \node {-2}; & \node {-1}; & \node {0};&\node {1}; &\node {2}; &\\
      \node{{\em (4)}};&& \node {-9}; &\node {-8}; &\node {-7}; & \node {-6}; & \node {-5}; & \node {-4}; & \node {-3}; & \node {-2}; & \node {-1};&\node {0}; &\node {1}; &\\
      \node{{\em (5)}};&& \node {-10}; & \node {-9}; &\node {-8}; &\node {-7}; & \node {-6}; & \node {-5}; & \node {-4}; & \node {-3}; & \node {-2}; & \node {-1};&\node {0}; &\\
      &&&&&&& \node{$\vdots$}; &&&&&& \\
    };
  \end{tikzpicture} 
\end{center}
}
\end{example}

Suppose that $t$ is an infinite frieze.  Following the notations of Conway and Coxeter \cite{CC73-1, CC73-2} and Holm and J{\o}rgensen \cite{HJ13}, we will set, for each $i,j\in\mathbb{Z}$:
\begin{itemize}
\item $a_i=t(i-1, i+1)$,
\item $f_i=t(-1, i)$,
\item $g_i=t(0,i)$,
\item $c_{ij}=\left|\begin{array}{cc} t(i,k) & t(i, k+1) \\ t(j,k) & t(j, k+1)\end{array}\right|$, for some $k\in\mathbb{Z}$,
\item $d_{ij}=\left|\begin{array}{cc} t(k,i) & t(k, j) \\ t(k+1,i) & t(k+1, j)\end{array}\right|$, for some $k\in\mathbb{Z}$.
\end{itemize}
At this point, it is important to observe that the numbers $c_{ij}$ and $d_{ij}$ are independent of the choice of $k$ by \cite[(Proposition 11.2)]{R12}. Notice also that the $f_i$'s and the $g_i$'s form the $-1$-th and $0$-th rows of the infinite frieze, respectively, while the $a_i$'s form the diagonal above the diagonal of ones, as depicted in the illustrative figure below. 

\begin{center}
  \begin{tikzpicture}[auto]

    \matrix
    {
			&&\node{{\em (-4)}};&\node{{\em (-3)}};&\node{{\em (-2)}};&\node{{\em (-1)}};& \node{{\em (0)}}; &\node{{\em (1)}};&\node{{\em (2)}};&\node{{\em (3)}};&\node{{\em (4)}};& \\
      &&&&& \node{$\vdots$}; &&&& \\
%
%
      \node{{\em (-4)}};&&  \node {0}; & \node {1}; & \node {$a_{-3}$}; & \node {$*$}; & \node {$*$}; & \node {$*$}; & \node {$*$}; & \node {$*$}; & \node {$*$};  \\
      \node{{\em (-3)}};&&  \node {-1}; & \node {0}; & \node {1}; & \node {$a_{-2}$}; & \node {$*$}; & \node {$*$}; & \node {$*$}; & \node {$*$}; & \node {$*$};  \\
      \node{{\em (-2)}};&&  \node {$-a_{-3}$}; & \node {-1}; & \node {0}; & \node {1}; & \node {$a_{-1}$}; & \node {$*$}; & \node {$*$}; & \node {$*$}; & \node {$*$};  \\
      \node{{\em (-1)}};&&  \node {$f_{-4}$}; & \node {$f_{-3}=-a_{-2}$}; & \node {$f_{-2}=-1$}; & \node {$f_{-1}=0$}; & \node {$f_0=1$}; & \node {$f_1=a_0$}; & \node {$f_2$}; & \node {$f_3$}; & \node {$f_4$};  \\
      \node{{\em (0)}};&\node{$\cdots$}; & \node {$g_{-4}$}; & \node {$g_{-3}$}; & \node {$g_{-2}=-a_{-1}$}; & \node {$g_{-1}=-1$}; & \node {$g_0=0$}; & \node {$g_1=1$}; & \node {$g_2=a_1$}; & \node {$g_3$}; & \node {$g_4$}; & \node{$\cdots$};\\
      \node{{\em (1)}};&& \node {$*$}; & \node {$*$}; & \node {$*$}; & \node {$-a_0$}; & \node {1}; & \node {0}; & \node {1}; & \node {$a_2$}; & \node {$*$};  \\
      \node{{\em (2)}};&&  \node {$*$}; & \node {$*$}; & \node {$*$}; & \node {$*$}; & \node {$-a_1$}; & \node {-1}; & \node {0}; & \node {1}; & \node {$a_3$};  \\
      \node{{\em (3)}};&&  \node {$*$}; & \node {$*$}; & \node {$*$}; & \node {$*$}; & \node {$*$}; & \node {$-a_2$}; & \node {-1}; & \node {0}; & \node {1};  \\
      \node{{\em (4)}};&&  \node {$*$}; & \node {$*$}; & \node {$*$}; & \node {$*$}; & \node {$*$}; & \node {$*$}; & \node {$-a_3$}; & \node {-1}; & \node {0};  \\
%
%
      &&&&&& \node{$\vdots$}; &&&&& \\
    };

  \end{tikzpicture} 
\end{center}

Clearly, every infinite frieze is uniquely determined by the sequence $(a_i)_{i\in\mathbb{Z}}$ via the unimodular rule.  Using the terminology of \cite{BPT15}, we will call this sequence the \textbf{quiddity sequence} of the infinite frieze.

Our first aim is to show that each number $t(i,j)$ can be expressed in terms of the $f_i$'s and $g_j$'s; compare with (17) and (32) in \cite{CC73-1,CC73-2}.

\begin{lemma}\label{lem:position}
Let $t$ is an infinite frieze.  Then, for all $p,q\in\mathbb{Z}$,
\[
t(
p,q)=\left|\begin{array}{cc} f_p & f_q \\ g_p & g_q\end{array}\right|.
\]
\end{lemma}
\begin{proof}
In \cite[Proposition 5.7]{HJ13}, it is demonstrated that, for all $i,j,p,q\in\mathbb{Z}$ satisfying $i<j$ and $p<q$, we have 
\[
\left|\begin{array}{cc} t(i,p) & t(i,q) \\ t(j,p) & t(j,q)\end{array}\right|=c_{ij}d_{pq}.
\]
Actually, the proof is straightforward and one can easily verify that the assumptions $i<j$ and $p<q$ are not necessary for the proof, and neither is the positivity of the entries. Therefore this relation holds for all $i,j,p,q\in \mathbb{Z}$.  Taking $k=p$ in the definition of $c_{ij}$ and $d_{ij}$, one then gets
\[
\left|\begin{array}{cc} t(i,p) & t(i,q) \\ t(j,p) & t(j,q)\end{array}\right|=\left|\begin{array}{cc} t(i,p) & t(i,p+1) \\ t(j,p) & t(j,p+1)\end{array}\right|\cdot \left|\begin{array}{cc} t(p,p) & t(p,q) \\ t(p+1,p) & t(p+1,q)\end{array}\right|.
\]
Now, letting $i=-1$ and $j=0$ gives
\[
\begin{array}{ccl}
\left|\begin{array}{cc} f_p & f_q \\ g_p & g_q\end{array}\right| &=&\left|\begin{array}{cc} f_p & f_{p+1} \\ g_p & g_{p+1}\end{array}\right|\cdot \left|\begin{array}{cc} t(p,p) & t(p,q) \\ t(p+1,p) & t(p+1,q)\end{array}\right|\\[4mm]
& = & 
\left|\begin{array}{cc} f_p & f_{p+1} \\ g_p & g_{p+1}\end{array}\right|\cdot \left|\begin{array}{cc} 0 & t(p,q) \\ -1 & t(p+1,q)\end{array}\right|\\ [3mm]
& = & 
t(p,q)
\end{array}
\]
since $\displaystyle\left|\begin{array}{cc} f_p & f_{p+1} \\ g_p & g_{p+1}\end{array}\right|=1$ by the unimodular rule.
\end{proof}

Consequently, an infinite frieze is determined as long as two adjacent rows are filled up.  The next two results show that the two rows do not need to be adjacent.

\begin{proposition}\label{prop:ptolemy}
Let $t$ be an infinite frieze.  For all $i, j, p, q \in\mathbb{Z}$, we have
\[
t(i,p)t(j,q)=t(i,j)t(p,q)+t(i,q)t(j,p).
\]
\end{proposition}
\begin{proof}
It is elementary to verify (compare with \cite[(Remark 5.3)]{HJ13}) that
\[
\left|\begin{array}{cc} f_i & f_p \\ g_i & g_p\end{array}\right|\cdot \left|\begin{array}{cc} f_j & f_q \\ g_j & g_q\end{array}\right| = \left|\begin{array}{cc} f_i & f_j \\ g_i & g_j\end{array}\right|\cdot \left|\begin{array}{cc} f_p & f_q \\ g_p & g_q\end{array}\right| + \left|\begin{array}{cc} f_i & f_q \\ g_i & g_q\end{array}\right|\cdot \left|\begin{array}{cc} f_j & f_p \\ g_j & g_p\end{array}\right|.
\]
Consequently, it follows from Lemma \ref{lem:position} that
\[
t(i,p)t(j,q)=t(i,j)t(p,q)+t(i,q)t(j,p),
\]
\end{proof}

\begin{corollary}\label{cor:fraction}
Let $t$ be an infinite frieze.  Let $i,j\in\mathbb{Z}$ be such that $i\neq j$. For all $p,q\in\mathbb{Z}$, we have:
\[
t(p,q)=\frac{t(i,p)t(j,q)-t(i,q)t(j,p)}{t(i,j)}.
\]
In particular, the infinite frieze is uniquely determined by its $i$-th and $j$-th rows when $i\neq j$.
\end{corollary}
\begin{proof}
This follows directly from Proposition \ref{prop:ptolemy} as $i\neq j$ if and only if $t(i,j)\neq 0$ by definition of an infinite frieze.
\end{proof}

\begin{remark}
As a complement to the above corollary, let us mention that a similar approach, using the columns instead of the rows, leads to
\[
t(p,q)=\frac{t(p,i)t(q,j)-t(p,j)t(q,i)}{t(i,j)}
\]
so that the $i$-th and $j$-th columns uniquely determine the infinite frieze when $i\neq j$.
\end{remark}

The following result, obtained in \cite[Lemma 1.7]{T15}, is an easy consequence of Proposition \ref{prop:ptolemy}.  We omit its proof.

\begin{corollary}[Compare with (18) and (24) in \cite{CC73-1, CC73-2}] \label{cor:recurrence}
Let $t$ is an infinite frieze.  
\begin{enumerate}
\item[(a)] For all $p,q\in\mathbb{Z}$, we have 
\[
t(p,q)a_q-t(p,q-1)=t(p,q+1).
\]
In particular, $a_qf_q-f_{q-1}=f_{q+1}$ for all $q\in \mathbb{Z}$.
\item[(b)] If $p,q\in\mathbb{Z}$, with $q\geq p+2$, then
\[
t(p,q)=\left|\begin{array}{cccccc}
a_{p+1} & 1 & 0 &\cdots & 0 & 0\\
1 & a_{p+2} & 1  & \cdots & 0 & 0\\
0 & 1 & a_{p+3} & \cdots & 0 & 0 \\
\vdots & \vdots & \vdots& \ddots& \vdots& \vdots\\
0 & 0 & 0 & \cdots & a_{q-2} & 1\\
0 & 0 & 0 & \cdots & 1 & a_{q-1}
\end{array}\right|
\]
In particular, if $q\geq 1$,
\[
f_q=\left|\begin{array}{cccccc}
a_{0} & 1 & 0 &\cdots & 0 & 0\\
1 & a_{1} & 1  & \cdots & 0 & 0\\
0 & 1 & a_{2} & \cdots & 0 & 0 \\
\vdots & \vdots & \vdots& \ddots& \vdots& \vdots\\
0 & 0 & 0 & \cdots & a_{q-2} & 1\\
0 & 0 & 0 & \cdots & 1 & a_{q-1}
\end{array}\right|.
\]
\end{enumerate}
\end{corollary}
%
%

\begin{remark}
It follows from part (a) of the above corollary that
\[
a_s=\frac{f_{s-1}+f_{s+1}}{f_s}
\]
for all $s\neq -1$ as $f_s=0$ if and only if $s=-1$. Consequently, $f_s$ divides $f_{s-1}+f_{s+1}$ for all $s\neq -1$. Moreover, since an infinite frieze is uniquely determined by its quiddity sequence $(a_i)_{i\in\mathbb{Z}}$, it follows that an infinite frieze is uniquely determined by the $f_i$'s and $a_{-1}$, or more generally by its $s$-th row and $a_s$. Unlike the case of frieze patterns, an infinite frieze is not uniquely determined by the $f_i$'s.  The knowledge of $a_{-1}$ is crucial. Indeed, consider the following infinite frieze:
{\em
\begin{center}
  \begin{tikzpicture}[auto]

    \matrix
    {
			&&\node{{\em (-5)}};&\node{{\em (-4)}};&\node{{\em (-3)}};&\node{{\em (-2)}};&\node{{\em (-1)}};&\node{{\em (0)}};&\node{{\em (1)}};& \node{{\em (2)}}; &\node{{\em (3)}};&\node{{\em (4)}};&\node{{\em (5)}};& \\
      &&&&&&& \node{$\vdots$}; &&&&&& \\
			 \node{{\em (-5)}}; && \node {0}; & \node {1}; & \node {2}; & \node {3}; & \node {4}; & \node {9}; & \node {14}; & \node {19
		};& \node {24};& \node {29};& \node {34};\\
			\node{{\em (-4)}}; && \node {-1}; &\node {0}; & \node {1}; & \node {2}; & \node {3}; & \node {7}; & \node {11}; & \node {15};& \node {19}; &\node {23}; &\node {27}; & \\
      \node{{\em (-3)}}; && \node {-2}; &\node {-1}; &\node {0}; & \node {1}; & \node {2}; & \node {5}; & \node {8}; & \node {11}; & \node {14}; &\node {17}; &\node {20}; &\\
      \node{{\em (-2)}};&& \node {-3}; &\node {-2}; &\node {-1}; & \node {0}; & \node {1}; & \node {3}; & \node {5}; & \node {7}; & \node {9}; &\node {11}; &\node {13}; &\\
      \node{{\em (-1)}};&& \node {-4}; &\node {-3}; &\node {-2}; & \node {-1}; & \node {0}; & \node {1}; & \node {2}; & \node {3}; & \node {4};&\node {5}; &\node {6}; &\\
      \node{{\em (0)}};&\node{$\cdots$};&\node {-9}; &\node {-7}; & \node {-5}; & \node {-3}; & \node {-1}; & \node {0}; & \node {1}; & \node {2}; & \node {3};&\node {4}; &\node {5}; &\node{$\cdots$};\\
      \node{{\em (1)}};&& \node {-14}; &\node {-11}; &\node {-8}; & \node {-5}; & \node {-2}; & \node {-1}; & \node {0}; & \node {1}; & \node {2};&\node {3}; &\node {4}; &\\
      \node{{\em (2)}};&& \node {-19}; &\node {-15}; &\node {-11}; & \node {-7}; & \node {-3}; & \node {-2}; & \node {-1}; & \node {0}; & \node {1};&\node {2}; &\node {3}; &\\
      \node{{\em (3)}};&& \node {-24}; &\node {-19}; &\node {-14}; & \node {-9}; & \node {-4}; & \node {-3}; & \node {-2}; & \node {-1}; & \node {0};&\node {1}; &\node {2}; &\\
      \node{{\em (4)}};&& \node {-29}; &\node {-23}; &\node {-17}; & \node {-11}; & \node {-5}; & \node {-4}; & \node {-3}; & \node {-2}; & \node {-1};&\node {0}; &\node {1}; &\\
      \node{{\em (5)}};&& \node {-34}; & \node {-27}; &\node {-20}; &\node {-13}; & \node {-6}; & \node {-5}; & \node {-4}; & \node {-3}; & \node {-2}; & \node {-1};&\node {0}; &\\
      &&&&&&& \node{$\vdots$}; &&&&&& \\
    };

  \end{tikzpicture} 
\end{center}
}
This infinite frieze has the same $-1$-th row (that is, the same $f_i$'s) as the infinite frieze of Example \ref{ex:tilings}, but is distinct from this latter.  Observe that both $a_{-1}=t(-2,0)$ are distinct.  Note that it follows from \cite[Corollary 2.2]{BPT15} that every choice of $a_{-1}$, with $a_{-1}\geq 2$, yields a different infinite frieze. 
\end{remark}


\section{Triangulations of the infinite strip}

In \cite{BPT15}, the authors considered the infinite strip of height one with marked points on the lower and upper boundaries, and showed that every infinite frieze can be realized as a triangulation of this strip.  However, this realization is not unique.  This is due, in part, to the existence of some asymptotic and generic arcs.  In this section, we adapt the definition of a triangulation of the infinite strip given in \cite{BPT15} and show, via an algorithm that is heavily inspired by Theorem 5.2 in \cite{BPT15}, that every triangulation of the infinite strip can be realized via such a triangulation.  This algorithm will play a crucial role in the proof of our main theorem. Finally, we introduce the Dehn twist equivalence and show that every entry in an infinite frieze can be obtained by using the CC-counting method and the BCI-counting method presented in Section~\ref{section:CC}.

\subsection{Triangulations of the strip}

Let $M_1=\mathbb{Z}$ and $M_2$ be a subset of $\mathbb{Z}$. Denote by $\mathbb{V}(M_1, M_2)$ the infinite strip of height one in the plane, with lower boundary $\{(x,0) \ | \ x\in\mathbb{R}\}$, upper boundary $\{(x,1) \ | \ x\in\mathbb{R}\}$, together with a set of marked points $\{(i,0) \ | \ i\in M_1\}$ on the lower boundary and a set of marked points $\{(i,1) \ | \ i\in M_2\}$ on the upper boundary.  By abuse of notation, we will sometimes identify these sets of marked points with $M_1$ and $M_2$. An \textbf{arc} in $\mathbb{V}(M_1, M_2)$ is a non-contractible curve whose endpoints are marked points from $M_1\cup M_2$, considered up to isotopy fixing endpoints. Two arcs are \textbf{compatible} if there are representatives in their isotopy classes that do not cross except, maybe, at their endpoints. We call an arc between marked points \textbf{bridging} (resp. \textbf{peripheral}) if its endpoints belong to different boundaries (resp. to the same boundary) of $\mathbb{V}(M_1, M_2)$.  Note that, by opposition to \cite{BPT15}, we do not consider here the limit points at the right and left extremities of the strip, nor the notions of generic and asymptotic arcs.

A \textbf{triangulation} of $\mathbb{V}(M_1, M_2)$ is a maximal (and necessarily infinite) collection of pairwise compatible arcs in $\mathbb{V}(M_1, M_2)$. Adapting the terminology of \cite{HJ12}, we say that a marked point $(k,i)$, with $i\in\{0,1\}$ (thus on the lower boundary or the upper boundary) of a triangulation of the strip is a {\bf left-fountain} if there exists an integer $k_0$ and infinitely many arcs in $T$ from $(k,i)$ to marked points to the left of $(j_0,0)$ or $(j_0,1)$. The notion of \textbf{right-fountain} is defined dually. By \cite{BPT15}, a triangulation of $\mathbb{V}(M_1, M_2)$ is called \textbf{admissible} if all marked points on the lower boundary are incident with only finitely many arcs, that is, if there is no left-fountain, nor right-fountain, on the lower boundary. Also, following \cite{GG14}, we say that a peripheral arc from $(i,0)$ to $(j,0)$ is \textbf{passing over} a set of marked points $(k_0,0), (k_1,0), \dots, (k_n,0)$ if $i\leq k_l\leq j$ for all $l\in\{1,2,\dots,n\}$.
Finally, following \cite{BCI74}, we say that a marked point is \textbf{special} if it is not incident to any arc in the triangulation.  Note that the boundary line segments joining two consecutive marked points on a same boundary of the strip $\mathbb{V}(M_1, M_2)$ are not considered as arcs of the triangulation.
\begin{example}
The triangulation $T_1$ of the infinite strip given by
\begin{center}
\includegraphics{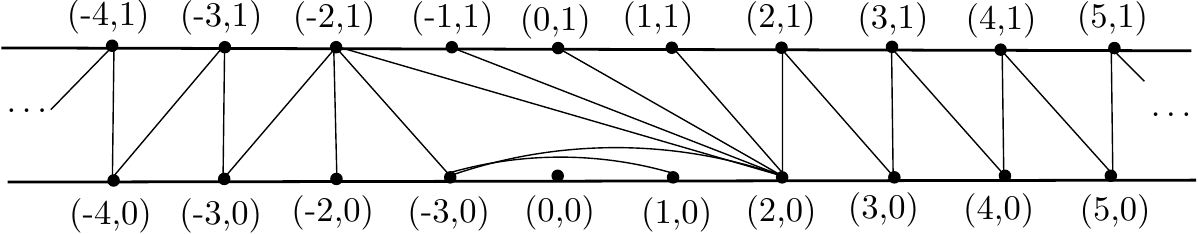}
\end{center}
is admissible with no marked point on the upper boundary, while the triangulation $T_2$ given by 
\begin{center}
\includegraphics{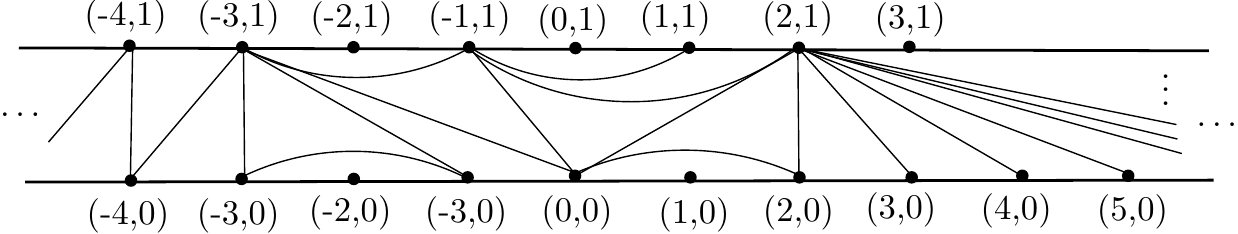}
\end{center}
is admissible, but as three special marked points on the upper boundary, namely $(-1,2)$, $(0,1)$ and $(3,1)$. Note that $(3,1)$ is a special marked point because $(2,1)$ is a right-fountain, and because $(3,1)$ is the rightmost marked point on the upper boundary.
\end{example}

We complete this section with two technical lemmata about the admissible triangulations that will be useful in the sequel.

\begin{lemma}\label{lem:bridging arcs}
Let $T$ be an admissible triangulation of the infinite strip $\mathbb{V}(M_1, M_2)$, for some subset $M_2$ of $\mathbb{Z}$. Assume moreover that there is no special marked points on the upper boundary. If $T$ contains a bridging arc, say from $(v,1)$ to $(j,0)$, then:
\begin{enumerate}
\item[(a)] There exists a bridging arc in $T$ from some $(u,1)$ to $(i,0)$, with $i<j$.
\item[(b)] There exists a bridging arc in $T$ from some $(u,1)$ to $(k,0)$, with $k>j$.
\end{enumerate}
\end{lemma}
\begin{proof}
(a). Since $T$ is admissible, $(j,0)$ is not a left-fountain.  Therefore, there exists a minimum integer $u$ such that there exists a bridging arc from $(u,1)$ to $(j,0)$ in $T$.  Clearly, $u\leq v$. Now, consider the set $P_j$ of all integers for which there exists a peripheral arc from $(m,0)$ to $(j,0)$ in $T$ and $m<j$.  Since $(j,0)$ is not a left-fountain, this set is either empty or as a minimum element, say $j_0$.  Set $i=j-1$ if $P_j=\varnothing$ and $i=j_0$ if $P_j\neq\varnothing$.  By construction, $i<j$ and the bridging arc from $(u,1)$ to $(i,0)$ does not cross any arc in $T$.  Consequently, this bridging arc is in $T$, and this shows the statement.

(b). The proof is dual.
\end{proof}

\begin{lemma}\label{lem:passing over}
Let $T$ be a triangulation of the infinite strip $\mathbb{V}(M_1, M_2)$, for some subset $M_2$ of $\mathbb{Z}$. Assume moreover that there is no special marked points on the upper boundary.  Then $T$ is admissible if and only if, for all $m,n\in\mathbb{Z}$ with $m<n$, there exists a peripheral arc passing over $(m,0)$ and $(n,0)$, or there exist integers $p$ and $q$ such that $p\leq m<n\leq q$, and bridging arcs from $(i,1)$ to $(p,0)$ and from $(j,1)$ to $(q,0)$ for some integers $i,j$ with $i\leq j$.
\end{lemma}
\begin{proof}
If the triangulation $T$ is not admissible, then $T$ contains at least one marked point $(k,0)$ incident with infinitely many arcs. Without loss of generality, suppose that $(k,0)$ is a left-fountain. Then, taking $m=k-1$ and $n=k+1$, it is easily understood that there is no peripheral arc passing over $(m,0)$ and $(n,0)$, nor there exist an integer $p$ with $p\leq m$ together with a bridging arc from $(u,1)$ to $(p,0)$ for some $u\in\mathbb{Z}$. This shows the sufficiency.  

For the necessity, let $m,n\in\mathbb{Z}$, with $m<n$. We prove the statement by induction on $d=m-n$.  Suppose that $d=1$. If $M_2=\varnothing$, the statement follows directly from \cite[Lemma 3.7]{GG14} (in which an admissible triangulation of $\mathbb{V}(M_1, \varnothing)$ is called locally finite). Suppose that $M_2\neq\varnothing$ and that there is no arc in $T$ passing over $(m,0)$ and $(n,0)$. Since $M_2\neq \varnothing$, and $T$ contains no special marked points on the upper boundary, there is at least one bridging arc in the triangulation.  Applying Lemma~\ref{lem:bridging arcs} repeatedly yields the result.


Suppose now, inductively, that $d>1$. Consider the marked points $(m+1,0)$ and $(n,0)$.  By induction hypothesis, there exists a peripheral arc in $T$, say $\alpha$, passing over $(m+1,0)$ and $(n,0)$, or there exist integers $p$ and $q$ such that $p\leq m+1<n\leq q$, and bridging arcs, say $\sigma$ and $\phi$, from $(u,1)$ to $(p,0)$ and from $(v,1)$ to $(q,0)$ respectively, for some integers $u,v$ with $u\leq v$. 

In the first case, if $\alpha$ passes over $(m,0)$ and $(n,0)$, the proof is over.  Else, $\alpha$ connects $(m+1,0)$ to $(n_0,0)$, for some $n_0\geq n$.  Since $(m+1,0)$ is not a right-fountain, we may furthermore assume that $n_0$ is maximal for this property. Consider the arc from $(m,0)$ to $(n_0,0)$.  If this arc is in $T$, the proof is over.  Otherwise, by the maximality of $n_0$, this implies that there is a bridging arc in $T$ ending at $(m+1,0)$. Applying Lemma~\ref{lem:bridging arcs} repeatedly yields the result.



In the second case, there exist integers $p$ and $q$ such that $p\leq m+1<n\leq q$, and bridging arcs, say $\sigma$ and $\phi$, from $(u,1)$ to $(p,0)$ and from $(v,1)$ to $(q,0)$ respectively, for some integers $u,v$ with $u\leq v$. Again, applying Lemma~\ref{lem:bridging arcs} repeatedly yields the result.

%
\end{proof}

\subsection{Constructions of Baur, Parsons and Tschabold revisited}\label{section:BPT}

In \cite[Theorem 5.6]{BPT15}, the authors show that every subset $M_2$ of $\mathbb{Z}$ and every admissible triangulation (with, possibly, a generic and asymptotic arcs) of the infinite strip $\mathbb{V}(M_1, M_2)$ gives rise to an infinite frieze.  The construction is simple: if $T$ is a triangulation of the infinite strip $\mathbb{V}(M_1, M_2)$ for some subset $M_2$ of $\mathbb{Z}$, let $a_i$ be the number of triangles incident with $(i,0)$ for every $i\in\mathbb{Z}$.  Then $(a_i)_{i\in\mathbb{Z}}$ is the quiddity sequence of an infinite frieze.  We will denote this infinite frieze by $\Phi(T)$. This construction, which mimics the idea of Corollary~\ref{cor:CC_quiddity}, still applies to our context. 

The construction of Baur, Parsons and Tschabold, see \cite[Theorem 5.2]{BPT15}, associating to each infinite frieze $t$ a triangulation of the strip $\mathbb{V}(M_1, M_2)$, for some subset $M_2$ of $\mathbb{Z}$, is more complicated. Below, we rephrase this construction via an algorithm. Note however that the following algorithm has been slightly adapted so that no generic or asymptotic arcs are required, by opposition to what is implicitly proposed in \cite{BPT15}. Note that we do not explain here why the algorithm gives rise to triangulations of the strip; and refer to \cite{BPT15} for the explanations. 

Note however that Step (A) or Step (B) in the algorithm could never terminate, in which case the set of all arcs added via these steps yield a triangulation of the infinite strip, as described in Step (C).  In all cases, the resulting triangulation is denoted by $\Psi(t)$. 

\begin{algorithm}\label{algo}
\emph{
Let $t:\mathbb{Z}\times\mathbb{Z}\rightarrow \mathbb{Z}$ be an infinite frieze, with quiddity sequence $(a_i)_{i\in\mathbb{Z}}$. Set, temporarily, $k:=0$ and $a^{(k)}_i=a_i$ for all $i\in\mathbb{Z}$. 
\begin{enumerate}
	\item[(A)] Set $Z_k:=\{i \in\mathbb{Z} \ | \ a^{(k)}_i=1\}$. 
	\begin{itemize}
		\item If $Z_k= \varnothing$, then go to (B).
		\item Else, for every $i\in \mathbb{Z}$, let $i^{-}_k=\max\{j\in \mathbb{Z} \ | \ a_j^{(k)}\neq 0 \text{ and } j<i\}$ and $i^{+}_k=\min\{j\in \mathbb{Z} \ |\ a_j^{(k)}\neq 0 \text{ and } j>i\}$. Moreover, for every $i\in Z_k$,
				\begin{enumerate}
						\item[(1)]  Add a peripheral arc in $\Psi(t)$ from $(i^-_k,0)$ to $(i^+_k,0)$.  
						\item[(2)]
							Replace the sequence $(a^{(k)}_i)_{i\in\mathbb{Z}}$ by the sequence $(a^{(k+1)}_i)_{i\in\mathbb{Z}}$, where 
							\[
							a^{(k+1)}_i=\left\{\begin{array}{ll}
							0 & \text{if } i\in Z_k,\\
							a^{(k)}_i-2 & \text{if } i^-_k\in Z_k \text { and } i^+_k\in Z_k,\\
							a^{(k)}_i-1 & \text{if, either } i^-_k\in Z_k \text{ or } i^+_k\in Z_k, \\
							a^{(k)}_i & \text{if } i^-_k\notin Z_k \text { and } i^+_k\notin Z_k.\end{array}\right.  
							\]
						\item[(3)] Increment $k$ by one, that is, set $k:=k+1$, and repeat (A).
				\end{enumerate}
	\end{itemize}
	\item[(B)] By construction, for every $i\in\mathbb{Z}$, we have $a^{(k)}_i=0$ or $a^{(k)}_i\geq 2$. Moreover, for every 			$i\in\mathbb{Z}$ such that $a^{(k)}_i\geq 2$, there is no arc in $\Psi(t)$ passing over $(i,0)$ and $a^{(k)}_i-1$ arcs incident to $(i,0)$ are missing in the triangulation $\Psi(t)$ under construction. 
	For each $i\in\mathbb{Z}$, set
	\[
	d_i=\left\{\begin{array}{ll}
	0 & \text{if } a^{(k)}_i=0,\\
	a^{(k)}_i-2 & \text{if } a^{(k)}_i\geq 2.
	\end{array}\right.  
	\]
	Let $N:=1+\displaystyle\sum_{i\in\mathbb{Z}}d_i$.  Note that $N$ could be finite or infinite.  We add the above-mentioned missing arcs as follows. Note however that the following steps will terminate if and only if $N$ is finite (see Step (C) for the details).
	\begin{itemize}
		\item If $N=1$, then, for every $i\in\mathbb{Z}$, we have $a^{(k)}_i=0$ or $a^{(k)}_i= 2$.  In this case, add one marked 				point on the upper boundary, say $(1,1)$, and add one bridging arc in $\Psi(t)$ from $(1,1)$ to $(i,0)$, for every 									$i\in\mathbb{Z}$ such that $a^{(k)}_i= 2$. Then go to Step (C).
		\item Else, $N>1$. Let $i_0\in\mathbb{Z}$ be such that $a^{(k)}_{i_0}> 2$. Add $a^{(k)}_{i_0}-1$ consecutive marked points 				on the upper boundary and add bridging arcs in $\Psi(t)$ from these marked points to $(i_0, 0)$. Set, 						temporarily, $l:=0$. Then do (1) and (2).
			\begin{enumerate}
				\item[(1)] Let $i_{l+1}=\min\{i\in\mathbb{Z} \ | \ i> i_l \text{ and } a^{(k)}_{i}> 2\}$.  
				\begin{itemize}
				\item If no such $i_{l+1}$ exists, add a bridging arc in $\Psi(t)$ from the rightmost marked point on the upper boundary to all $(i,0)$ such that 									$i>i_l$ and $a^{(k)}_i= 2$; then set $l:=0$ and go to (2).  
				\item If such an $i_{l+1}$ exists, add one bridging arc in $\Psi(t)$ from the rightmost marked point on the upper boundary to $(i_{l+1}, 0)$. Then add $a^{(k)}_{i_{l+1}}-2$, consecutive marked points on the upper boundary, to the right of the already existing marked points, and add bridging arcs in $\Psi(t)$ from those marked points to $(i_{l+1},0)$.  Increment $l$ by one, that is, set $l:=l+1$, and repeat (1).
				\end{itemize}
				\item[(2)] Let $i_{l-1}=\max\{i\in\mathbb{Z} \ | \ i< i_l \text{ and } a^{(k)}_{i}> 2\}$.  
				\begin{itemize}
				\item If no such $i_{l-1}$ exist, add a bridging arc in $\Psi(t)$ from the leftmost marked point on the upper boundary to all $(i,0)$ such that $i<i_l$ 					and $a^{(k)}_i= 2$; then go to Step (C).  
				\item If such an $i_{l-1}$ exists, add one bridging arc in $\Psi(t)$ from the leftmost marked point on the upper boundary to $(i_{l-1}, 0)$. Then add $a^{(k)}_{i_{l-1}}-2$, consecutive marked 							points on the upper boundary, to the left of the already existing marked points, and add bridging arcs in $\Psi(t)$ from those marked 							points to $(i_{l-1},0)$.  Decrement $l$ by one, that is, set $l:=l-1$, and repeat (2).
				\end{itemize}
			\end{enumerate}
	\end{itemize}
	\item[(C)] The set of peripheral arcs, added in Step (A), together with the bridging arcs, added in Step (B), form an admissible 							triangulation of the infinite strip $\mathbb{V}(M_1, M_2)$, denoted by $\Psi(t)$, where $M_2$ is in bijection with
					\[
					\left\{\begin{array}{ll} \varnothing & \text{if (A) did not terminate}, \\
					\{1,2,\dots,N\} & \text{if (A) and (B) terminated, and $N<\infty$ in (B)},\\
					\mathbb{Z} & \text{if (A) terminated but (B)(1) and (B)(2) did not},\\
					-\mathbb{N}=\{\dots, -2, -1, 0\} & \text{if (A) and (B)(1) terminated but (B)(2) did not},\\
					\mathbb{N}=\{0,1,2,\dots\} & \text{if (A) and (B)(2) terminated but (B)(1) did not}.
					\end{array}\right.
					\] 					
					Moreover, for every $i\in\mathbb{Z}$, there are $a_i$ triangles incident to the vertex $(i,0)$. 
	\end{enumerate}
}
\end{algorithm}

\begin{remark}\label{rem:cond}
The algorithm provides admissible triangulations of the infinite strip having:
\begin{itemize}
\item No generic or asymptotic arcs, in the sense of \cite{BPT15},
\item No special marked point on the upper boundary.  In particular, by \cite[(Lemma 1)]{BCI74}, the triangulations contain no peripheral arcs on the upper boundary.
\end{itemize}
\end{remark}

\begin{example}\label{ex:triangulation}
Consider the quiddity sequence \[(a_i)_{i\in\mathbb{Z}}=(\dots, 3,3,4,2,1,6,2,2,\dots)\]  of an infinite frieze $t$. Without loss of generality, assume that \[\dots, a_{-5}^{(0)}=3, a_{-4}^{(0)}=3, a_{-3}^{(0)}=4, a_{-2}^{(0)}=2, a_{-1}^{(0)}=1, a_0^{(0)}=6, a_1^{(0)}=2, a_2^{(0)}=2, \dots.\]  
Initially, $Z_0:=\{-1\}$, since $a_{-1}^{(0)}=1$. Thus, applying Step (A), we draw a peripheral arc from $(-2,0)$ to $(0,0)$. Moreover, 
\[\dots, a_{-5}^{(1)}=3, a_{-4}^{(1)}=3, a_{-3}^{(1)}=4, a_{-2}^{(1)}=1, a_{-1}^{(1)}=0, a_0^{(1)}=5, a_1^{(1)}=2, a_2^{(1)}=2, \dots.\] 
We get:
\begin{center}
\includegraphics{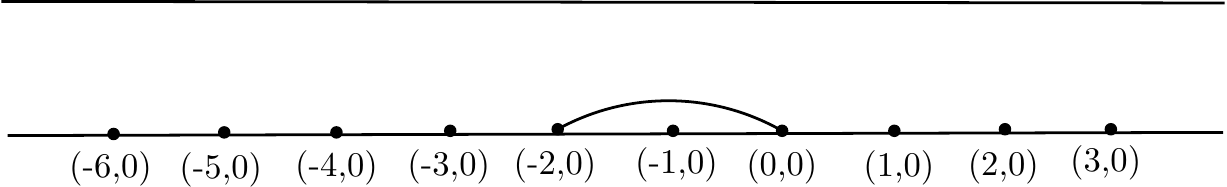}
\end{center}
Now $Z_1:=\{-2\}$, since $a_{-2}^{(1)}=1$. Thus, applying Step (A) again, we draw a peripheral arc from $(-3,0)$ to $(0,0)$. Moreover,   
\[\dots, a_{-5}^{(2)}=3, a_{-4}^{(2)}=3, a_{-3}^{(2)}=3, a_{-2}^{(2)}=0, a_{-1}^{(2)}=0, a_0^{(2)}=4, a_1^{(2)}=2, a_2^{(2)}=2, \dots.\] 
Thus:
\begin{center}
\includegraphics{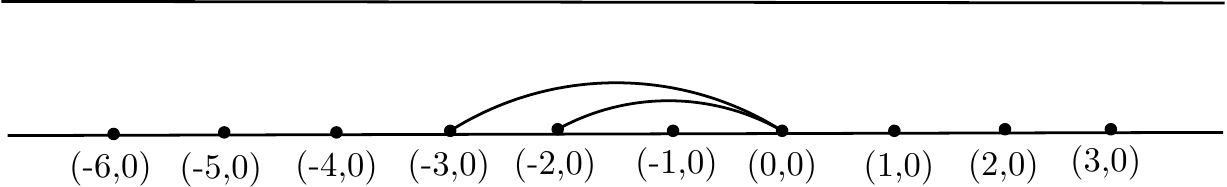}
\end{center}
At this point, $Z_2=\varnothing$, and thus Step (A) is over. A simple calculation shows that $N=\infty$. Let, say, $i_0=0$.  Since $a_0^{(2)}=4$, we draw $a_0^{(2)}-1=3$ marked points on the upper boundary, and we add bridging arcs from these marked points to $(0,0)$.
\begin{center}
\includegraphics{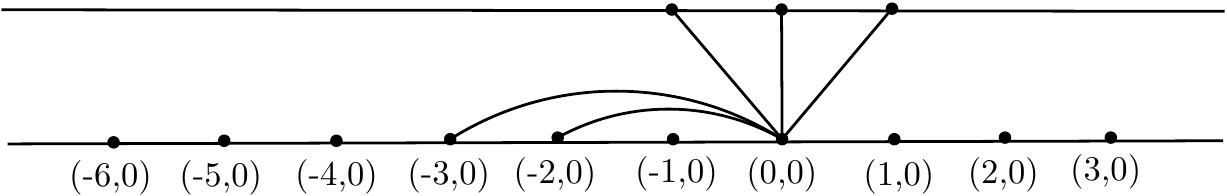}
\end{center}
We now apply Step (B)(1). Since $a_i^{(2)}=2$ for all $i>0$, $i_1$ does not exist. Consequently, we add a bridging arc from the rightmost marked point on the upper boundary to all $(i,0)$, with $i>0$.
\begin{center}
\includegraphics{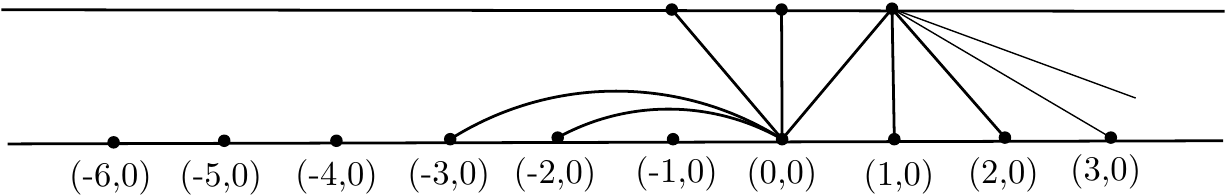}
\end{center}
Finally, for Step (B)(2), we have $i_{-1}=-3$. Thus we add a bridging arc from the leftmost marked point on the upper boundary to $(-3,0)$. Moreover, since $a_{-3}^{(2)}=3$, we add $a_{-3}^{(2)}-2=1$ marked point on the upper boundary and add a bridging arc from this marked point to $(-3,0)$.  Repeating Step (B)(2) indefinitely yields the final admissible triangulation $\Psi(t)$ given by
\begin{center}
\includegraphics{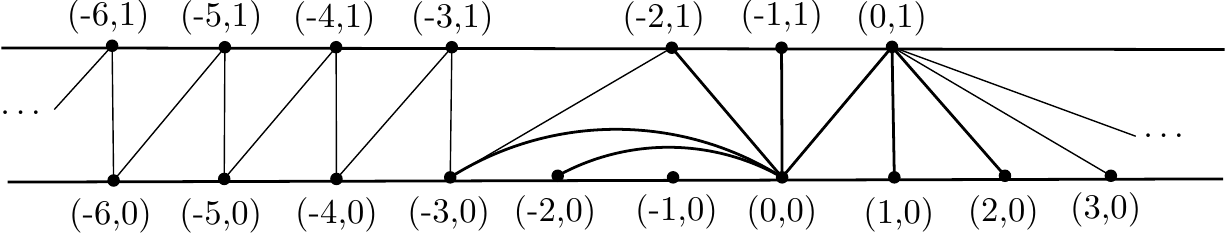}
\end{center}
where the marked points on the upper boundary were labeled in order to correspond with the conclusion in Step (C) concerning $M_2$.
\end{example}

\subsection{Dehn twist}

Let $M_2=\mathbb{Z}$, and suppose that $T$ is a triangulation of the infinite strip $\mathbb{V}(M_1, M_2)$ with no special marked point on the upper boundary. 
Following the terminology of \cite{V15}, we say that the \textbf{positive Dehn twist} of $T$ is the triangulation $D(T)$ whose arcs are obtained from those of $T$ by replacing every bridging arc from $(u,1)$ to $(j,0)$ in $T$ by a bridging arc from $(u+1, 1)$ to $(j, 0)$, but keeping the peripheral arc of $T$ on the lower boundary.
Similarly, \textbf{negative Dehn twist} of $T$ is the triangulation $D^{-1}(T)$ whose arcs are obtained from those of $T$ by replacing every bridging arc from $(u,1)$ to $(j,0)$ in $T$ by a bridging arc from $(u-1, 1)$ to $(j, 0)$, but keeping the peripheral arc of $T$ on the lower boundary.

Finally, let $D^n(T)$ be the triangulation obtained from $T$ upon applying $n$ times the positive Dehn twist, if $n\geq 0$, or $-n$ times the negative Dehn twist if $n<0$.

\begin{definition}
Two triangulations $T$ and $T'$ of $\mathbb{V}(M_1, \mathbb{Z})$ are called \textbf{Dehn twist equivalent} if $T'$ can be obtained from $T$ upon a sequence of positive or negative Dehn twists, that is, $T'=D^{n}(T)$ for some $n\in\mathbb{Z}$.
\end{definition}


\begin{example}
If $T$ is the triangulation of the infinite strip given by
\begin{center}
\includegraphics{triangulation.pdf}
\end{center}
then its positive Dehn twist $D(T)$ is given by
\begin{center}
\includegraphics{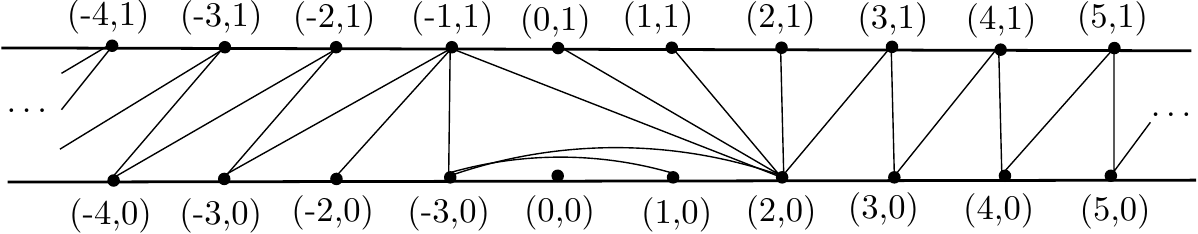}
\end{center}
\end{example}

Observe that if two triangulations $T$ and $T'$ of the infinite strip are Dehn equivalent, then, for every $i\in\mathbb{Z}$, the number of triangles incident to the marked point $(i,0)$ is the same for both triangulations.  Consequently, $\Phi(T)=\Phi(T')$.
\subsection{CC-counting and BCI-counting methods}

Let $T$ be an admissible triangulation of the infinite strip, for some subset $M_2$ of $\mathbb{Z}$. The purpose of this section is to show that the CC-counting and BCI-counting methods presented in Section~\ref{section:CC} still applies to the infinite frieze $\Phi(T)$, that is, every entry in the infinite frieze $\Phi(T)$ can be obtained via the CC-counting or the BCI-counting method applied on the marked points of the infinite strip.  Observe that the fact that $T$ is admissible makes it possible to compute $CC_T(i,j)$ and $BCI_T(i,j)$ for any two marked points on the lower boundary of the infinite strip, with $i\leq j$.

\begin{proposition}\label{prop:CC counting}
Let $T$ be an admissible triangulation of the infinite strip $\mathbb{V}(M_1, M_2)$, for some subset $M_2$ of $\mathbb{Z}$, with no special marked point on its upper boundary.  Let $a_i$ denote the number of triangles adjacent to the marked point $(i,0)$ for each $i\in\mathbb{Z}$.  The infinite frieze $t=\Phi(T)$ having $(a_i)_{i\in\mathbb{Z}}$ has quiddity sequence satisfies $t(i,j)=CC_T(i,j)=BCI_T(i, i+1, \dots,j)$ for all $i,j\in\mathbb{Z}$ with $i\leq j$.
\end{proposition}
\begin{proof}
Let $i,j\in\mathbb{Z}$.  If $i=j$, then $t(i,j)=0=CC_T(i,j)$ by convention.  Similarly, if $j=i+1$, then $t(i,j)=1=CC_T(i,j)$.  Suppose that $j\geq i+2$.  By Lemma~\ref{lem:passing over}, there exists a peripheral arc passing over $(i-1,0)$ and $(j+1,0)$, or there exist integers $p$ and $q$ such that $p\leq i-1<j+1 q$, and bridging arcs from $(u,1)$ to $(p,0)$ and from $(v,1)$ to $(q,0)$ for some integers $u,v$ with $u\leq v$.  Since $T$ is admissible with no special marked points on its upper boundary, cutting along the peripheral arc or the two bridging arcs yields, in either case, a triangulation of a polygon $P$ having finitely many vertices.  Moreover, since this cutting does not cut off any of the triangles incident to $(k,0)$, with $k\in\{i, i+1, \dots, j\}$, there are still $a_k$ triangles adjacent to the marked point $(k,0)$ in $P$, for all $k\in\{i, i+1, \dots, j\}$. Consequently, we have $CC_T(i,j)=CC_P(i,j)$ and $BCI_T(i,i+1, \dots, j)=BCI_P(i,i+1, \dots, j)$, and thus $CC_T(i,j)=BCI_T(i,i+1, \dots, j)$ by Proposition~\ref{prop:BCI}.  Finally, it follows from \cite{CC73-1,CC73-2} and Corollary~\ref{cor:recurrence} that \[
t(i,j)=\left|\begin{array}{cccccc}
a_{i+1} & 1 & 0 &\cdots & 0 & 0\\
1 & a_{i+2} & 1  & \cdots & 0 & 0\\
0 & 1 & a_{i+3} & \cdots & 0 & 0 \\
\vdots & \vdots & \vdots& \ddots& \vdots& \vdots\\
0 & 0 & 0 & \cdots & a_{j-2} & 1\\
0 & 0 & 0 & \cdots & 1 & a_{j-1}
\end{array}\right|=CC_P(i,j).
\]
\end{proof}

The following useful corollary follows immediately from the preceding proposition and the CC-counting method.

\begin{corollary}\label{cor:peripheral}
Let $T$ be an admissible triangulation of the infinite strip $\mathbb{V}(M_1, M_2)$, for some subset $M_2$ of $\mathbb{Z}$, with no special marked point on its upper boundary.  Let $t=\Phi(T)$.  For all $i,j\in\mathbb{Z}$, with $i\leq j$, we have $t(i,j)=1$ if and only if $j=i+1$ or there is a peripheral arc in $T$ from $(i,0)$ to $(j,0)$.
\end{corollary}

\section{Main results}

In this section, we show our main results:
\begin{theorem}\label{thm:main}
Algorithm~\ref{algo} provides a bijection between the infinite friezes and the admissible triangulations of the infinite strip  $\mathbb{V}(M_1, M_2)$ having no special marked points on the upper boundary, up to Dehn twist equivalence when $M_2=\mathbb{Z}$.
\end{theorem} 

\subsection{Proof of Theorem \ref{thm:main}}

Observe first that if $t$ is an infinite frieze with quiddity sequence $(a_i)_{i\in\mathbb{Z}}$, then the admissible triangulation of the strip $\Psi(t)$ provided by Algorithm~\ref{algo} has $a_i$ triangles incident to vertex $(i,0)$.  Thus, because an infinite frieze is uniquely determined by its quiddity sequence, we get $\Phi(\Psi(t))=t$. Consequently, the construction $\Phi$, when restricted to the admissible triangulations with no special marked points on their upper boundaries, remains surjective. 

The objective of the following lemmata is to show the injectivity of $\Phi$, up to Dehn twist equivalence in the case where $M_2=\mathbb{Z}$, and thus that $\Phi$ and $\Psi$ are inverse constructions.

\begin{lemma}\label{lem:N not M}
Let $T$ and $T'$ be admissible triangulations with no special marked points on their upper boundaries of the infinite strip $\mathbb{V}(M_1, M_2)$ and $\mathbb{V}(M_1, M_2')$ respectively, for some subsets $M_2$ and $M_2'$ of $\mathbb{Z}$. If $\Phi(T)=\Phi(T')$, then there is a bijection between $M_2$ and $M_2'$.  
\end{lemma}
\begin{proof}
Let $t=\Phi(T)=\Phi(T')$ and suppose that $(a_i)_{i\in\mathbb{Z}}$ is the quiddity sequence of $t$.  By definition of $\Phi$, for each $i\in\mathbb{Z}$, there are $a_i$ triangles incident to $(i,0)$ in $T$ and $T'$, and thus $a_i-1$ arcs incident to $(i,0)$. Moreover, by Corollary~\ref{cor:peripheral}, $T$ and $T'$ have the same peripheral arcs (on their lower boundaries). So only the bridging arcs of $T$ and $T'$ could differ, although the quantity of bridging arcs is the same at each marked point $(i,0)$.  
There are four possibilities:
\begin{enumerate}
\item[(1)] If $M_2=\varnothing$, then there is no bridging arc in $T$, and thus no bridging arc in $T'$.  So $|M_2|=0=|M_2'|$.  
\item[(2)] Suppose that $0<|M_2|<\infty$. In this case, set $N=|M_2|$, and label the marked points on the upper boundary by $(1,1), (2,1), \dots, (N,1)$, where $(1,1)$ is the leftmost marked point and $(N,1)$ is the rightmost marked point. Since $T$ is admissible with no special marked points on the upper boundary, it follows from Lemma~\ref{lem:bridging arcs} that $(1,1)$ is a left-fountain.  Dually, $(N,1)$ is a right-fountain. 

Now, let $i,j$ be any two integers such that $(i,0)$ and $(j,0)$ are connected to $(1,1)$ via bridging arcs, and $i<j$.  We will show that, in $T'$, the marked points $(i,0)$ and $(j,0)$ are also connected to the same marked point on the upper boundary.  Indeed, suppose that, in $T'$, $(i,0)$ is connected to $(u_1,1)$ and $(j,0)$ is connected to $(v,1)$, with $u_1\leq v$.  Now, the triangulation $T$ induces a triangulation of the polygon $P$ with vertices $\{(i,0), (i+1, 0), \dots (j,0), (1,1)\}$.  Similarly $T'$ induces a triangulation of the polygon $P'$ with vertices $\{(i,0), (i+1, 0), \dots (j,0), (v,1), (v-1, 1), \dots, (u_1,1)\}$. Observe that $P'$ has more vertices than $P$ if $u_1<v$, and thus its triangulation requires more arcs than the one of $P$. However, since the peripheral arcs on the lower boundary are the same in $P$ and $P'$, and that there is no special marked points on the upper boundaries (thus no peripheral arcs on the upper boundaries), there must be more bridging arcs in $P'$ than in $P$, contradicting the fact that the number of bridging arcs incident to each marked point on the lower boundaries is the same in $T$ and $T'$. Thus $u_1=v$. Since this is true for any two marked points on the lower boundary that is connected to $(1,1)$ in $T$, this shows that $(u_1,1)$ is a left-fountain in $T'$.  Consequently, $(u_1,1)$ is the leftmost marked point on the upper boundary of $T'$.  Similarly, the fact the $(N,1)$ is a right-fountain in $T$ implies the existence of a right-fountain on the upper boundary of $T'$, say $(u_{N'},1)$.  Moreover,  $(u_{N'},1)$ is the rightmost marked point on the upper boundary of $T'$.  Denote by $(u_1,1), (u_2,1), \dots, (u_{N'},1)$ the marked points on the upper boundary of $T'$, so that $|M_2'|=N'$.

Since the same argument can be used to show that any two marked points on the lower boundary that are connected to $(u_1,1)$ in $T'$ are connected to the same left-fountain in $T$, we get that the marked points on the lower boundary that are connected to $(1,1)$ in $T$ coincide with those that are connected to $(u_1, 1)$ in $T'$.  Similarly, the marked points on the lower boundary that are connected to $(N,1)$ in $T$ coincide with those that are connected to $(u_{N'}, 1)$ in $T'$. Let $(m_0,0)$ be the rightmost marked point on the lower boundary that is connected to $(1,1)$ in $T$ (and $(u_1,1)$ in $T'$).  Similarly, let $(n_0,0)$ be the leftmost marked point on the lower boundary that is connected to $(N,1)$ in $T$ (and $(u_{N'},1)$ in $T'$).  Then $T$ induces a triangulation on the polygon $Q$ with vertices $\{(1,1), \dots, (N,1), (n_0,0), (n_0-1, 0), \dots (m_0,0)\}$.  Similarly, $T'$ induces a triangulation of the polygon $Q'$ with vertices $\{(u_1,1), \dots, (u_{N'},1),(n_0,0), (n_0-1, 0), \dots (m_0,0)\}$.  By the same argument as for the triangulated polygons $P$ and $P'$ above, one must have $N=N'$.  So $|M_2|=|M_2'|$.
\item[(3)] Suppose that $M_2$ is in bijection with $\mathbb{N}$. Observe that $M_2'$ is infinite, since otherwise the cases (1) and (2) applied to $M_2'$ would imply that $M_2$ is finite, a contradiction.  Moreover, since $M_2$ is in bijection with $\mathbb{N}$, then $T$ has a leftmost marked point on the upper boundary.  Just as in case (2), this implies that $T'$ also has a leftmost marked point on its upper boundary. Therefore, $T'$ is in bijection with $\mathbb{N}$.  Dually, if $M_2$ is in bijection with $-\mathbb{N}$, then so does $M_2'$.
\item[(4)] Suppose that $M_2$ is in bijection with $\mathbb{Z}$.  By cases (1) to (3) applied to $M_2'$, $M_2'$ is infinite and $T'$ has no leftmost nor rightmost marked points on its upper boundary.  Consequently,  $M_2'$ is in bijection with $\mathbb{Z}$.
\end{enumerate}
\end{proof}

\begin{lemma}\label{lem:uniqueness}
Let $T$ and $T'$ be admissible triangulations with no special marked points on their upper boundaries of the infinite strip $\mathbb{V}(M_1, M_2)$, for some subset $M_2$ and $M_2'$ of $\mathbb{Z}$. If $M_2$ is a proper subset of $\mathbb{Z}$, then $T=T'$.
\end{lemma}
\begin{proof}
As in the proof of the preceding lemma, let $t=\Phi(T)=\Phi(T')$ and suppose that $(a_i)_{i\in\mathbb{Z}}$ is the quiddity sequence of $t$.  By definition of $\Phi$, for each $i\in\mathbb{Z}$, there are $a_i$ triangles incident to $(i,0)$ in $T$ and $T'$, and thus $a_i-1$ arcs incident to $(i,0)$. Moreover, by Corollary~\ref{cor:peripheral}, $T$ and $T'$ have the same peripheral arcs (on their lower boundaries). So only the bridging arcs of $T$ and $T'$ could differ, although the quantity of bridging arcs is the same at each marked point $(i,0)$.  
\begin{enumerate}
\item[(1)] If $M_2=\varnothing$, then there is no bridging arc in $T$, and thus no bridging arc in $T'$.  So $T=T'$.  
\item[(2)] Suppose that $0<|M_2|<\infty$. Let $N=|M_2|$, and let $(1,1), (2,1), \dots, (N,1)$ be the marked points on the upper boundary. By the proof of case (2) in Lemma~\ref{lem:N not M}, $(1,1)$ and $(N'1)$ are left and right-fountains, respectively, for both $T$ and $T'$.  Moreover, the marked points on the lower boundary that are connected to these fountains via bridging arcs are the same in $T$ and $T'$.  If $N=1$, there is then nothing to show. Inductively, suppose that $N>1$.  For simplicity, suppose, without loss of generality, that the rightmost marked point on the lower boundary that is connected to $(1,1)$ (in both $T$ and $T'$) is $(1,0)$. 
Under this assumption, the bridging arc from $(1,0)$ to $(2,1)$ does not cross any arcs in $T$ and $T'$, and therefore is in $T$ and $T'$. This reduces the situation to a case with $N-1$ marked points on the upper boundary, namely $(2,1), (3,1), \dots, (N,1)$.  By induction, $T=T'$.
\item[(3)] Suppose that $M_2$ is in bijection with $\mathbb{N}$, and let $(1,1), (2,1), \dots$ denote the marked points on the upper boundary, where $(1,1)$ is the leftmost one, in both $T$ and $T'$. Moreover, as in (2), the marked points on the lower boundary that are connected to $(1,1)$ via bridging arcs are the same in $T$ and $T'$. Now suppose, inductively, that $n\geq 2$ and that the marked points on the lower boundary that are connected to $(n-1, 1)$ are the same in $T$ and $T'$.  Since there is no special marked points on the upper boundaries in $T$ and $T'$, there is a bridging arc incident to $(n,1)$.  Therefore, there is a rightmost marked point on the lower boundary connected to $(n-1,1)$, and this marked point is the same in $T$ and $T'$, say $(1,0)$ for simplicity. As in the proof of case (2) in Lemma~\ref{lem:N not M}, we show that the marked points on the lower boundary that are connected to $(n,1)$ are the same in $T$ and $T'$.  First, since $(1,0)$ is the rightmost marked point on the lower boundary connected to $(n-1,1)$ in $T$ and $T'$, there must be a bridging arc from $(1,0)$ to $(n,1)$ in both $T$ and $T'$. Then, suppose that $i$ is any integer such that $(i,0)$ is connected to $(n,1)$ via a bridging arc in $T$.  Observe that $i\geq 1$.  Suppose that, in $T'$, $(i,0)$ is connected to $(u,1)$, with $u\geq n$.  Now, the triangulation $T$ induces a triangulation of the polygon $P$ with vertices $\{(1,0), (2, 0), \dots (i,0), (n,1)\}$.  Similarly $T'$ induces a triangulation of the polygon $P'$ with vertices $\{(1,0), (2, 0), \dots (i,0), (u,1), (u-1, 1), \dots, (n,1)\}$. Observe that $P'$ has more vertices than $P$ if $n<u$, and thus its triangulation requires more arcs than the one of $P$. However, since the peripheral arcs on the lower boundary are the same in $P$ and $P'$, and that there is no special marked points on the upper boundaries (thus no peripheral arcs on the upper boundaries), there must be more bridging arcs in $P'$ than in $P$, contradicting the fact that the number of bridging arcs incident to each marked point on the lower boundaries is the same in $T$ and $T'$. Thus $u=n$. Since one could start with the bridging arcs in $T'$ instead of those in $T$, this shows that the marked points on the lower boundary that are connected to $(n,1)$ are the same in $T$ and $T'$.  By induction, $T=T'$.

Clearly, if $M_2$ is in bijection with $-\mathbb{N}$, the argument is dual.
\end{enumerate}
\end{proof}

\begin{lemma}\label{lem:dehn}
Let $T$ and $T'$ be admissible triangulations with no special marked points on their upper boundaries of the infinite strip $\mathbb{V}(M_1, M_2)$, with $M_2=\mathbb{Z}$. Then $T$ and $T'$ are Dehn twist equivalent.
\end{lemma}
\begin{proof}
Let $t$ be an infinite frieze with quiddity sequence $(a_i)_{i\in\mathbb{Z}}$.  Suppose that $t=\Phi(T)$, for some admissible triangulation of the infinite strip $\mathbb{V}(M_1, M_2)$, with $M_2=\mathbb{Z}$, with no special marked point on the upper boundary. By definition of $\Phi$, for each $i\in\mathbb{Z}$, there are $a_i$ triangles incident to $(i,0)$ in $T$, and thus $a_i-1$ arcs incident to $(i,0)$. Moreover, by Corollary~\ref{cor:peripheral}, the peripheral arcs on the lower boundary of $T$ are completely determined by $t$, so that the quantity of bridging arcs ending at each marked point is also determined by $t$.

Denote by $M$ the set of marked points on the lower boundary that are incident to at least one bridging arcs in $T$. Observe that because $|M_2|=\infty$ and $T$ is admissible with no special marked points on its upper boundary, the set $M$ is also infinite. Moreover, it follows from Lemma~\ref{lem:bridging arcs} that $M$ is in bijection with $\mathbb{Z}$.  So let $M=\{(m_i,0) \ | \ i\in\mathbb{Z}\}$. Suppose furthermore that $(m_i,0)$ is incident to $b_i$ bridging arcs. 

Consider $m_0\in M$.  Then $(m_0, 0)$ is incident to $b_0$ bridging arcs in $T$.  Observe that since $T$ has no special marked point on its upper boundary, and thus no peripheral arc on its upper boundary, then the upper endpoints of these bridging arcs are all consecutive, say $(0,1), (1,1), \dots (b_0,1)$, where $(0,1)$ is the leftmost one. We claim that this choice of establishing $(0,1)$ has the leftmost marked point on the upper boundary that is connected to $(m_0,0)$ completely determines $T$.  Indeed, consider the marked point $(m_1,0)$.  For $T$ to be a triangulation, the bridging arc from $(m_1,0)$ to $(b_0,1)$ must be in $T$. Moreover, as for $(m_0, 0)$, $(m_1, 0)$ must be connected to $b_1$ consecutive marked points on the upper boundary, namely $(b_0, 1), (b_0+1, 1), \dots (b_0+b_1-1, 1)$.  Continuing in this way with $(m_2, 0), (m_3, 0), \dots$, and similarly with $(m_{-1},0), (m_{-2}, 0), \dots$ show that $T$ is completely determined by the choice of $(1,0)$ as the leftmost marked point connected to $(m_0,0)$. 

Clearly, a different choice for the leftmost marked point connected to $(m_0,0)$ would yield a different triangulation $T'$, in which the upper endpoints of all bridging arcs are shifted to the right or to the left compared to those in $T$.  In other words, $T$ and $T'$ would be Dehn twist equivalent.
\end{proof}

Combining the above three lemmata, we get the proof of the theorem.  Indeed, let $T$ and $T'$ be admissible triangulations with no special marked points on their upper boundaries of the infinite strips $\mathbb{V}(M_1, M_2)$ and $\mathbb{V}(M_1, M_2')$ respectively, for some subsets $M_2$ and $M_2'$ of $\mathbb{Z}$.  If $\Phi(T)=\Phi(T')$, then there is a bijection between $M_2$ and $M_2'$ by Lemma~\ref{lem:N not M}. Consequently, up to labeling of the marked points on the upper boundaries, we have $M_2=M_2'$.  If $M_2\neq \mathbb{Z}$, it follows from Lemma~\ref{lem:uniqueness} that $T=T'$. Else, $M_2=\mathbb{Z}$, and it follows from Lemma~\ref{lem:dehn} that $T$ and $T'$ are Dehn twist equivalent. Consequently, $\Phi$ is injective, up to Dehn twist equivalence when $M_2=\mathbb{Z}$. This completes the proof.

\subsection{Triangulations of $\mathbb{V}(M_1, \varnothing)$.}

The triangulations of the $\infty$-gon were first studied by Holm and Jorgensen in \cite{HJ12} in connection with the study of a cluster structure on a certain category of infinite Dynkin type.  This category, and the triangulations of the $\infty$-gon, were then further investigated by Grabowski and Gratz \cite{GG14}. In this setting, the {\bf $\infty$-gon} is interpreted as the discrete line, that is the line of integers, and a triangulation of the $\infty$-gon is a maximal collection of peripheral arcs, in the sense described previously.  Consequently, a triangulation of the $\infty$-gon is equivalent to a triangulation of the infinite strip $\mathbb{V}(M_1, \varnothing)$ with no special marked points on its upper boundary.

\begin{example}\label{ex:infty-gon}
A possible admissible triangulation of the $\infty$-gon is given by
\begin{center}
\includegraphics{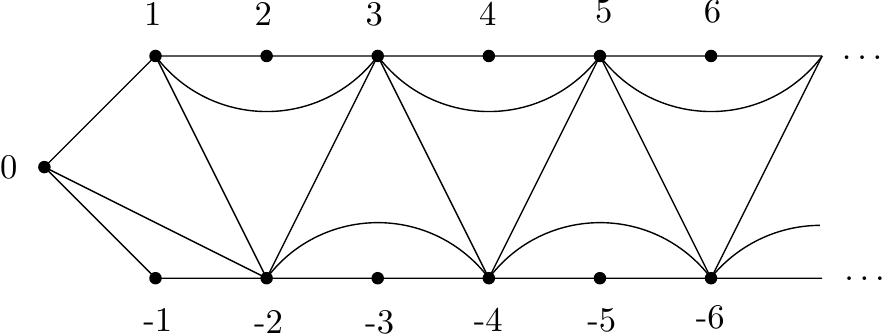}
\end{center}
where we \lq\lq bended\rq\rq\, the discrete line in order to make the pattern of the triangulation more evident for further purposes.
\end{example}

In \cite{HJ13}, Holm and J{\o}rgensen showed that the admissible triangulations of the infinite strip $\mathbb{V}(\mathbb{Z}, \mathbb{Z})$ (with no special marked point on the upper boundary) are in bijection with the $SL_2$-tilings of the discrete plane with enough ones.  In this section, we adapt the terminology of \cite{HJ13} to our context to show a similar result in connection with the admissible triangulation of the $\infty$-gon, equivalently, of the infinite strip $\mathbb{V}(M_1, \varnothing)$ with no special marked points on the upper boundary.

\begin{definition}
An infinite frieze $t:\mathbb{Z}\times\mathbb{Z}\rightarrow\mathbb{Z}$ is said to \textbf{have enough ones} if, for all $(i,j)\in\mathbb{Z}\times\mathbb{Z}$ with $i\leq j$, there exists $(i', j')\in\mathbb{Z}\times\mathbb{Z}$ such that $i'\leq i\leq j\leq j'$ and $t(i',j')=1$. 
\end{definition}

\begin{proposition}\label{prop:enough ones}
The admissible triangulations of the infinite strip $\mathbb{V}(M_1, \varnothing)$ with no special marked point on the upper boundary are in bijection with the infinite friezes with enough ones.
\end{proposition}
\begin{proof}
Let $T$ be an admissible triangulation of the infinite strip $\mathbb{V}(M_1, \varnothing)$ with no special marked point on the upper boundary.  Since there are no bridging arcs, it follows from Lemma~\ref{lem:passing over} that for all $m, n\in\mathbb{Z}$
there exists a peripheral arc passing over $(m, 0)$ and $(n, 0)$. Now consider the infinite frieze $\Phi(T)$ introduced in Section~\ref{section:BPT}.  It follows from Corollary~\ref{cor:peripheral} that $\Phi(T)$ has enough ones.

Conversely, suppose that $t$ is an infinite frieze with enough ones, and suppose that $t=\Phi(T)$, where $T$ is an admissible triangulations of the infinite strip $\mathbb{V}(M_1, M_2)$ with no special marked point on the upper boundary.  By Corollary~\ref{cor:peripheral} again, there is a peripheral arc passing over any two marked points $(m,0)$ and $(n,0)$.  Consequently, there cannot be bridging arcs in $T$.  Since $T$ has no special marked point on the upper boundary, we get $M_2=\varnothing$.  The bijection then follows from Theorem~\ref{thm:main}. 
\end{proof}

\begin{example}
The following infinite frieze has enough ones since it is the infinite frieze corresponding to the triangulation of the $\infty$-gon given in Example~\ref{ex:infty-gon}.  Observe that $\dots, a_{-2}=1, a_{-1}=3, a_0=2, a_1=1, a_2=5, a_3=1, \dots$.
\end{example}
{\em
\begin{center}
  \begin{tikzpicture}[auto]
    \matrix
    {
			&&\node{{\em (-5)}};&\node{{\em (-4)}};&\node{{\em (-3)}};&\node{{\em (-2)}};&\node{{\em (-1)}};& \node{{\em (0)}}; &\node{{\em (1)}};&\node{{\em (2)}};&\node{{\em (3)}};&\node{{\em (4)}};&\node{{\em (5)}};& \\
      &&&&&&& \node{$\vdots$}; &&&&&& \\
      \node{{\em (-5)}}; && \node {0}; & \node {1}; & \node {5}; & \node {4}; & \node {15}; & \node {11}; & \node {7}; & \node {10}; & \node {3}; & \node {5}; & \node {2}; \\
      \node{{\em (-4)}};&& \node {-1}; & \node {0}; & \node {1}; & \node {1}; & \node {4}; & \node {3}; & \node {2}; & \node {3}; & \node {1}; & \node {2}; & \node {1}; \\
      \node{{\em (-3)}};&& \node {-5}; & \node {-1}; & \node {0}; & \node {1}; & \node {5}; & \node {4}; & \node {3}; & \node {5}; & \node {2}; & \node {5}; & \node {3}; \\
      \node{{\em (-2)}};&& \node {-4}; & \node {-1}; & \node {-1}; & \node {0}; & \node {1}; & \node {1}; & \node {1}; & \node {2}; & \node {1}; & \node {3}; & \node {2}; \\
      \node{{\em (-1)}};&& \node {-15}; & \node {-4}; & \node {-5}; & \node {-1}; & \node {0}; & \node {1}; & \node {2}; & \node {5}; & \node {3}; & \node {10}; & \node {7}; \\
      \node{{\em (0)}};&\node{$\cdots$}; & \node {-11}; & \node {-3}; & \node {-4}; & \node {-1}; & \node {-1}; & \node {0}; & \node {1}; & \node {3}; & \node {2}; & \node {7}; & \node {5}; &\node{$\cdots$};\\
      \node{{\em (1)}};&& \node {-7}; & \node {-2}; & \node {-3}; & \node {-1}; & \node {-2}; & \node {-1}; & \node {0}; & \node {1}; & \node {1}; & \node {4}; & \node {3}; \\
      \node{{\em (2)}};&& \node {-10}; & \node {-3}; & \node {-5}; & \node {-2}; & \node {-5}; & \node {-3}; & \node {-1}; & \node {0}; & \node {1}; & \node {5}; & \node {4}; \\
      \node{{\em (3)}};&& \node {-3}; & \node {-1}; & \node {-2}; & \node {-1}; & \node {-3}; & \node {-2}; & \node {-1}; & \node {-1}; & \node {0}; & \node {1}; & \node {1}; \\
      \node{{\em (4)}};&& \node {-5}; & \node {-2}; & \node {-5}; & \node {-3}; & \node {-10}; & \node {-7}; & \node {-4}; & \node {-5}; & \node {-1}; & \node {0}; & \node {1}; \\
      \node{{\em (5)}};&& \node {-2}; & \node {-1}; & \node {-3}; & \node {-2}; & \node {-7}; & \node {-5}; & \node {-3}; & \node {-4}; & \node {-1}; & \node {-1}; & \node {0}; \\
      &&&&&&& \node{$\vdots$}; &&&&&& \\
    };
  \end{tikzpicture} 
\end{center}
}

We complete our paper with the following observations.

\begin{lemma}\label{lem:special}
Let $T$ be a triangulation of the infinite strip $\mathbb{V}(M_1, M_2)$, for some subset $M_2$ of $\mathbb{Z}$.  If $T$ contains a peripheral arc on its lower boundary, then there exists at least one special vertex on its lower boundary.
\end{lemma}
\begin{proof}
Suppose that there is a peripheral arc from $(p,0)$ to $(q,0)$, with $p<q$.  Consider the polygon $P$ formed by the vertices $(p,0), (p+1,0), \dots, (q,0)$ together with its triangulation inherited from $T$.  If $q=p+2$, then $P$ is a triangle, and the vertex $p+1$ is special in $P$, and in $T$.  Else, $P$ has at least four vertices, and it is easy to show by induction, that $P$ has at least two special non-consecutive vertices (see \cite[Lemma 1]{BCI74}).  Consequently, at least one of the vertices $(p+1,0), (p+2,0), \dots, (q-1,0)$ is special in $P$, and thus in $T$.
\end{proof}

\begin{corollary}\label{prop:1}
Let $t$ be an infinite frieze with quiddity sequence $(a_i)_{i\in\mathbb{Z}}$. If $t$ has enough ones, then $a_i=1$ for some $i\in\mathbb{Z}$.
\end{corollary}
\begin{proof}
Suppose that $t$ is an infinite frieze with quiddity sequence $(a_i)_{i\in\mathbb{Z}}$, with enough ones. By Proposition~\ref{prop:enough ones}, there exists a triangulation $T$ of $\mathbb{V}(M_1, \varnothing)$ such that $t=\Phi(T)$.  Since $T$ has no bridging arcs, it follows from Lemma~\ref{lem:special} that $T$ has at least one special vertex on its lower boundary, say $(k,0)$.  Consequently, $a_k=1$.
\end{proof}

Observe that the above corollary also follows directly from Algorithm~\ref{algo}.  Indeed, if $a_i\geq 2$ for all $i\in\mathbb{Z}$, then there is nothing to do in Step (A).  Consequently, all arcs appearing in the triangulation $\Psi(t)$ are added in Step (B), and therefore are bridging arcs.  Thus, $M_2\neq \varnothing$.

\section*{Acknowledgements}
This work was initiated during Summer 2015 following an invitation from T. Holm, P. J{\o}rgensen and S. Morier-Genoud to the mini-workshop entitled \lq\lq Friezes\rq\rq\, in Oberwolfach (Germany) in November 2015. Fruitful discussions with K. Baur and M. Tschabold during and after the mini-workshop allowed this paper to take its much improved current form.  I am extremely grateful to these five people. The author was supported by the Natural Sciences and Engineering Research Council of Canada (NSERC).


\begin{thebibliography}{99}
%
\bibitem{ARS10} I. Assem, C. Reutenauer and D. Smith, {\em Friezes}, Adv. Math. {\bf 225} (2010), no. 6, 3134--3165.
%
\bibitem{BPT15} K. Baur, M. J. Parsons and M. Tschabold, {\em Infinite friezes}, preprint, 2015,  arXiv:1504.02695.
%
\bibitem{BHJ15} C. Bessenrodt, T. Holm and P. J{\o}rgensen, {\em All $SL_2$-tilings come from triangulations}, in prepapration.
%
\bibitem{BCI74} D. Broline, D. W. Crowe and I. M. Isaacs, {\em The geometry of frieze patterns}, Geometriae Dedicata {\bf 3} (1974), 171--176. 
%
\bibitem{CC73-1} J. H. Conway and H. S. M. Coxeter, {\em Triangulated polygons and frieze patterns}, Math. Gaz. {\bf 57} (1973), no. 400, 87--94.
%
\bibitem{CC73-2} J. H. Conway and H. S. M. Coxeter, {\em Triangulated polygons and frieze patterns}, Math. Gaz. {\bf 57} (1973), no. 401, 175--183.
%
\bibitem{GG14} J. E. Grabowski and S. Gratz, {\em Cluster algebras of infinite rank}, J. Lond. Math. Soc. (2) {\bf 89} (2014), no. 2, 337--363.
%
\bibitem{HJ13} T. Holm and P. J{\o}rgensen, {\em $SL_2$-tilings and triangulations of the strip}, J. Combin. Theory Ser. A {\bf 120} (2013), no. 7, 1817--1834. 
%
\bibitem{HJ12} T. Holm and P. J{\o}rgensen, {\em On a cluster category of infinite Dynkin type, and the relation to triangulations of the infinity-gon}, Math. Z. {\bf 270} (2012), no. 1-2, 277--295.
%
\bibitem{MG15} S. Morier-Genoud, {\em Coxeters frieze patterns at the crossroads of algebra,
geometry and combinatorics}, preprint, 2015, arXiv:1503.05049.
%
\bibitem{R12} C. Reutenauer, {\em Linearly recursive sequences and Dynkin diagrams}, preprint, 2012, arXiv:1204.5145.
%
\bibitem{T15} M. Tschabold, {\em Arithmetic infinite friezes from punctured discs}, preprint, 2015, arXiv:1503.04352.
%
\bibitem{V15} H. Vogel, {\em Asymptotic triangulations and Coxeter transformations of the annulus}, preprint, 2015, arXiv:1508.00485 
\end{thebibliography}
\end{document}